\documentclass[11pt]{amsart}
\usepackage{amsfonts, amssymb, amsmath, amsthm, hyperref, color, float,enumitem}
\usepackage{url}

\theoremstyle{plain}
   \newtheorem{teo}{Theorem}
   \newtheorem{coro}[teo]{Corollary}
   \newtheorem{lema}[teo]{Lemma}
   \newtheorem{propo}[teo]{Proposition}
   
\theoremstyle{definition}
   
\theoremstyle{remark}
 \newtheorem{obs}{Remark}

\numberwithin{equation}{section}
\allowdisplaybreaks
\newcommand{\R}{\mathbb{R}} 
\newcommand{\N}{\mathbb{N}} 
\newcommand{\norm}[1]{\left\|#1\right\|} 
\newcommand{\supp}{\mathrm{supp}}

\begin{document}

\title[Mixed inequalities for multilinear commutators]{Mixed inequalities for commutators with multilinear symbol}

\author[F. Berra]{Fabio Berra}
\address{CONICET and Departamento de Matem\'{a}tica (FIQ-UNL),  Santa Fe, Argentina.}
\email{fberra@santafe-conicet.gov.ar}

\author[M. Carena]{Marilina Carena}
\address{CONICET and Departamento de Matem\'{a}tica (FIQ-UNL),  Santa Fe, Argentina.}
\email{marilcarena@gmail.com}

\author[G. Pradolini]{Gladis Pradolini}
\address{CONICET and Departamento de Matem\'{a}tica (FIQ-UNL),  Santa Fe, Argentina.}
\email{gladis.pradolini@gmail.com}

\thanks{The authors were supported by CONICET, UNL and ANPCyT}



\begin{abstract}
	We prove mixed inequalities for commutators of Calder\'on-Zygmund operators (CZO) with multilinear symbols. Concretely, let $m\in\mathbb{N}$ and $\mathbf{b}=(b_1,b_2,\dots, b_m)$ be a vectorial symbol such that each component $b_i\in \mathrm{Osc}_{\mathrm{exp}\, L^{r_i}}$, with $r_i\geq 1$. If  $u\in A_1$ and $v\in A_\infty(u)$ we prove that the inequality
	\[uv\left(\left\{x\in \mathbb{R}^n: \frac{|T_\mathbf{b}(fv)(x)|}{v(x)}>t\right\}\right)\leq C\int_{\mathbb{R}^n}\Phi\left(\|\mathbf{b}\|\frac{|f(x)|}{t}\right)u(x)v(x)\,dx\]
	holds for every $t>0$, where $\Phi(t)=t(1+\log^+t)^r$, with $1/r=\sum_{i=1}^m 1/r_i$.
	
	We also consider operators of convolution type with kernels satisfying less regularity properties than CZO. In this setting, we give a Coifman type inequality for the associated commutators with multilinear symbol. This result allows us to deduce the $L^p(w)$-boundedness of these operators when $1<p<\infty$ and $w\in A_p$. As a consequence, we can obtain the desired mixed inequality in this context.
\end{abstract}

\maketitle

\section{Introduction}
The problem of characterizing the nonnegative functions $w$ for which the Hardy-Littlewood maximal operator $M$ is bounded in $L^p(w)$, for $1<p<\infty$, was first solved by Muckenhoupt in \cite{Muck72} and established the emergence of the well-known $A_p$ classes. 

Some years later, looking for an alternative proof of Muckenhoupt theorem, Sawyer showed in \cite{Sawyer} that the inequality
\[uv\left(\left\{x\in \mathbb{R}: \frac{M(fv)(x)}{v(x)}>t\right\}\right)\leq \frac{C}{t}\int_{\mathbb{R}} |f(x)|u(x)v(x)\,dx\]
holds if $u$ and $v$ are $A_1$ weights, for every positive $t$.
Notice that this inequality can be seen as a generalization of the weak $(1,1)$ type of $M$, which corresponds to the case $v=1$, characterized by $A_1$ weights.
Thus, Muckenhoupt theorem can be obtained by combining the inequality above with interpolation techiques and the Jones factorization theorem. We shall refer to these type of estimates as mixed inequalities because they include two different weights.

A further extension of this inequality was established and proved in \cite{CruzUribe-Martell-Perez}. More precisely, it was shown that if $u$ and $v$ satisfy either $u,v\in A_1$ or $u\in A_1$ and $v\in A_\infty(u)$, then the inequality
 \[uv\left(\left\{x\in \mathbb{R}^n: \frac{|\mathcal{T}(fv)(x)|}{v(x)}>t\right\}\right)\leq \frac{C}{t}\int_{\mathbb{R}^n} |f(x)|u(x)v(x)\,dx\]
holds for every positive $t$ where $\mathcal{T}$ is either $M$ or a Calder\'on-Zygmund operator (CZO). 
More general results were proved in \cite{L-O-P} and in \cite{Berra21} by considering weaker conditions on the weights given by $u\in A_1$ and $v\in A_\infty$. The former involves the operator $M$ and the latter,  the generalized maximal operator $M_\Phi$ associated to a Young function $\Phi$ with additional properties.

In the next sections we state our main results concerning to Calderón-Zygmund and Hörmander type operators, separately. We shall be dealing with a linear operator $T$, bounded on $L^2(\R^n)$ and such that for $f\in L^2$ with compact support we have the representation
\begin{equation}\label{eq: representacion integral del T}Tf(x)=\int_{\R^n}K(x-y)f(y)\,dy ,\quad\quad x\notin \supp f,\end{equation}
where $K$ is a measurable function defined away from the origin. Note that we are assuming this integral representation only for $x\not\in \supp f$.

\subsection*{Commutators of CZO}
Given a locally integrable function $b$ and an operator $T$ as in \eqref{eq: representacion integral del T}, the commutator of $T$ is denoted by $T_b$ or $[b,T]$ and defined by the expression
\begin{equation}\label{eq: conmutador de orden 1}
[b,T]f(x)=bTf(x)-T(bf)(x).
\end{equation}
For $m\in\mathbb{N}$, the higher order commutator of $T$ is given recursively by 
\[T_b^m f =\left[b, T_b^{m-1}f\right].\]

In \cite{Berra-Carena-Pradolini(M)} we proved mixed inequalities for higher order commutators of Calderón-Zygmund operators. More precisely, given $u\in A_1$, $v\in A_\infty(u)$, $b\in\mathrm{BMO}$, $m\in \mathbb{N}$ and a CZO $T$, the inequality
\begin{equation}\label{eq: desigualdad mixta para Tbm OCZ b en BMO}
uv\left(\left\{x\in \mathbb{R}^n: \frac{|T_b^m(fv)(x)|}{v(x)}>t\right\}\right)\leq C\int_{\mathbb{R}^n}\Phi\left(\frac{\|b\|_{\mathrm{BMO}}^m|f(x)|}{t}\right)u(x)v(x)\,dx
\end{equation}
holds for every positive $t$, where $\Phi(t)=t(1+\log^+t)^m$. Notice that condition $v\in A_\infty(u)$ implies that both $uv$ and $v$ belong to $ A_\infty$. In this case many classical tools of Harmonic Analysis, like Calder\'on-Zygmund decomposition, can be used to achieve the desired estimates. On the other hand, when $v=1$ inequality \eqref{eq: desigualdad mixta para Tbm OCZ b en BMO} was proved by Pérez in \cite{Perez95}.

In this paper we study mixed inequalities for commutators associated to a multilinear symbol $\mathbf{b}=(b_1,b_2,\dots,b_m)$, where each $b_i$ is a locally integrable function.  
These operators are defined by
\begin{equation}\label{eq: conmutador de orden m con simbolo multilineal}
T_{\mathbf{b}}f=\left[b_1, T_{\overline{\mathbf{b}}_{1}}f\right]. 
\end{equation}
where $\overline{\mathbf{b}}_{1}=(b_2,\dots,b_m)$ (see Section \ref{seccion: propiedades de Tb} for details). 
In fact, we shall prove that we can define $T_{\mathbf{b}}$ by commuting any component $b_i$ with the corresponding operator $T_{\overline{\mathbf{b}}_{i}}$ (see Corollary~\ref{coro: definicion recursiva de Tb independiente de la componente con que se conmuta} below). If $\mathbf{b}=(b_1)$ we simply write $T_{\mathbf{b}}=T_{b_1}$, as in \eqref{eq: conmutador de orden 1}.

We describe now the space of functions considered for the components of $\mathbf{b}$. Given $r\geq 1$ and $b\in L^1_{loc}$ we say that $b\in\mathrm{Osc}_{\mathrm{exp}\,L^r}$ if
\begin{equation}\label{eq: definicion de Osc_exp_L^r}
\|b\|_{\mathrm{Osc}_{\mathrm{exp}\,L^r}}=\sup_Q \|b-b_Q\|_{\mathrm{exp}\,L^r, Q}<\infty
\end{equation}
where $Q$ denotes any cube in $\mathbb{R}^n$ with sides parallel to the coordinate axes, $b_Q$ is the average of $b$ over $Q$ and $\|g\|_{\varphi,Q}$ denotes a $\varphi$-Luxemburg type average (see section below for details). Given $\mathbf{b}=(b_1,\dots,b_m)$ where $b_i\in \mathrm{Osc}_{\mathrm{exp}\,L^{r_i}}$ and $r_i\geq 1$ for every $1\leq i\leq m$, we will denote 
\[\|\mathbf{b}\|=\prod_{i=1}^m \|b_i\|_{\mathrm{Osc}_{\mathrm{exp}\,L^{r_i}}}.\]

We are now in a position to state our main result concerning to mixed inequalities for $T_{\mathbf{b}}$. 

\begin{teo}\label{teo: desigualdad mixta para Tb caso CZO}
	Let $m\in \mathbb{N}$, $r_i\geq 1$ for $1\leq i\leq m$ and $1/r=\sum_{i=1}^m1/r_i$. Let $\Phi(t)=t(1+\log^+t)^{1/r}$ and $\mathbf{b}=(b_1,b_2,\dots,b_m)$ a multilinear symbol such that each $b_i\in \mathrm{Osc}_{\mathrm{exp}\,L^{r_i}}$. If $u\in A_1$, $v\in A_\infty(u)$ and $T$ is a CZO, then there exists a positive constant $C$ such that the inequality
	\[uv\left(\left\{x\in \mathbb{R}^n: \frac{|T_\mathbf{b}(fv)(x)|}{v(x)}>t\right\}\right)\leq C\int_{\mathbb{R}^n}\Phi\left(\|\mathbf{b}\|\frac{|f(x)|}{t}\right)u(x)v(x)\,dx,\]
	holds for every $t>0$.
\end{teo}
Observe that if we consider $b_i=b\in \mathrm{Osc}_{\mathrm{exp}\,L^{1}} $ for every $1\leq i\leq m$, then we get \eqref{eq: desigualdad mixta para Tbm OCZ b en BMO}, since $\mathrm{Osc}_{\mathrm{exp}\,L^{1}}=\mathrm{BMO}$.

\subsection*{Commutators of operators of Hörmander type}
Let $T$ be a linear operator as in \eqref{eq: representacion integral del T}. We shall consider kernels $K$ with less regularity properties associated to a Young function $\varphi$, that will be called $H_{\varphi,m}$, where $m\in \mathbb{N}$ and $H_{\varphi,0}=H_{\varphi}$ (for further details see Section~\ref{seccion: preliminares}). 

In this section we state mixed inequalities involving the commutators of these type of operators with multilinear symbol $\mathbf{b}=(b_1, \dots, b_m)$. Our main result is contained in the following theorem.
\begin{teo}\label{teo: desigualdad mixta para Tb multilineal Hormander}
	Let $1<q<\infty$ and $q^2/(2q-1)<\beta<q$. Assume that $u\in A_1\cap \mathrm{RH}_s$ for some $s>1$ and $v^{\alpha}\in A_{(q/\beta)'}(u)$, with $\alpha=\beta(q-1)/(q-\beta)$. Let $m\in\mathbb{N}$, $r_i\geq 1$ for every $1\leq i\leq m$, $1/r=\sum_{i=1}^m 1/r_i$ and $\mathbf{b}=(b_1,b_2,\dots,b_m)$ where $b_i\in \mathrm{Osc}_{\mathrm{exp}\, L^{r_i}}$, for $1\leq i\leq m$. If $T$ is an operator with kernel $K\in H_{\phi,m}$,  where $\bar{\eta}^{-1}(t)\phi^{-1}(t)(\mathrm{log}\, t)^{1/r}\lesssim t$ for $t\geq e$ and $\bar\eta\in B_\rho$ for every $\rho\geq \min\{\beta, s\}$, then there exists $C>0$ such that the inequality
	\[uv\left(\left\{x\in\mathbb{R}^n: \left|\frac{T_{\mathbf{b}}(fv)(x)}{v(x)}\right|>t\right\}\right)\leq C\int_{\mathbb{R}^n} \Phi\left(\frac{\|\mathbf{b}\||f(x)|}{t}\right)u(x)v(x)\,dx\]
	holds for every positive $t$, and where $\Phi(t)=t(1+\log^+t)^{1/r}$.
\end{teo}
{\begin{obs}\label{obs: H_{phi,m} implica H_{eta,m}}
		Notice that the hypothesis $K\in H_{\phi,m}$ implies that $K\in H_{\eta,m}$, since we are assuming the relation
		\[\bar{\eta}^{-1}(t)\phi^{-1}(t)(\mathrm{log}\, t)^{1/r}\lesssim t \quad\textrm{ for } \quad t\geq e.\]
\end{obs}}

When we consider a single symbol $b \in \mathrm{BMO}$ the result above was proved by the authors in~\cite{Berra-Carena-Pradolini(M)}. The proof involves the Calderón-Zygmund decomposition combined with the boundedness of $T_b^m$ on $L^p(w)$ for $1<p<\infty$ and $w\in A_p$. As far as we know, this last result is unknown for the commutators of $T$ when $\mathbf{b}$ is a multilinear symbol with components belonging to the spaces $\mathrm{Osc}_{\mathrm{exp}\, L^r}$, $r\geq 1$. We prove this result in the next theorem.

\begin{teo}\label{teo: tipo fuerte de Tb Hormander}
	Let $m\in\mathbb{N}$, $r_i\geq 1$ for every $1\leq i\leq m$ and $1/r=\sum_{i=1}^m 1/r_i$. Let  $1\leq\beta<p$ and $\eta, \phi$ be two Young functions verifying $\bar{\eta}^{-1}(t)\phi^{-1}(t)(\mathrm{log}\, t)^{1/r}\lesssim t$, for every $t\geq e$. Let $T$ be an operator with kernel $K\in H_{\phi,m}$ and $\mathbf{b}=(b_1,b_2,\dots,b_m)$ where $b_i\in \mathrm{Osc}_{\mathrm{exp}\,L^{r_i}}$, for $1\leq i\leq m$. If $\bar\eta\in B_\rho$ for $\rho>\beta$ and $w^p\in A_{p/\beta}$, then there exists a positive constant $C$ such that
	\[\|(T_\mathbf{b}f)w\|_{L^p}\leq C\|\mathbf{b}\|\|fw\|_{L^p},\]
	provided the left-hand side is finite.
\end{teo}
When  $b$ is a single symbol belonging to BMO this result was proved in \cite{LRdlT}. On the other hand, if $T$ is a CZO and $\mathbf{b}$ as above an analogous result is contained in \cite{P-T}. The theorem above is a consequence of a Coifman type estimate that we shall state in Section~{\ref{seccion: caso Hormander}}.

The article is organized as follows. In Section~\ref{seccion: preliminares} we give the previous definitions and basic results. In Section~\ref{seccion: propiedades de Tb} we prove some technical facts concerning commutators with a multilinear symbol $\mathbf{b}$. Finally, in Section~\ref{seccion: caso CZO} and Section~\ref{seccion: caso Hormander} we prove the results for CZO and Hörmander type operators, respectively.

\section{Preliminaries and basic definitions}\label{seccion: preliminares}
In this section we provide the concepts and basic results required for the reading of this work. 

A weight $w$ is a locally integrable  function defined on
$\R^n$, such that $0<w(x)<\infty$ in a.e. $x\in \R^n$. For $1<p<\infty$ the Muckenhoupt $A_p$ class is defined as the set of all weights $w$
for which there exists a positive constant $C$ such that the inequality
\[\left(\frac{1}{|Q|}\int_Q w\right)\left(\frac{1}{|Q|}\int_Q w^{-\frac{1}{p-1}}\right)^{p-1}
\leq C\] 
holds for every cube $Q\subset \R^n$  with sides parallel
to the coordinate axes.  For $p=1$ we say that $w\in A_1$ if there
exists a positive constant $C$ such that
\begin{equation*}
\frac{1}{|Q|}\int_Q w\le C\, \inf_Q w(x),
\end{equation*}
for every cube $Q\subset \R^n$. %
The smallest constant $C$ for which the Muckenhoupt condition holds
is called the characteristic $A_p$-constant of $w$, and denoted by $[w]_{A_p}$.
The $A_\infty$ class is defined by the collection of all the $A_p$ classes.
It is easy to see
that if $p<q$ then $A_p\subseteq A_q$. For further properties and the basic theory of Muckenhoupt weights
we refer the reader to~\cite{javi} and~\cite{GC-RdF}.

Given weights $u$ and $w$ we say $w\in A_p(u)$, $1\leq p\leq \infty$, if the above inequalities hold with the Lebesgue measure replaced by $d\mu(x)=u(x)\,dx$. The corresponding characteristic constants are denoted by $[w]_{A_p(u)}$. The following result about $A_p(u)$ classes was proved in \cite{CruzUribe-Martell-Perez}.

\begin{lema}\label{lema: u en A1 y v en Ap(u) implica uv en Ap}
	If $1\leq p\leq \infty$, $u\in A_1$ and $v\in A_p(u)$, then $uv\in A_p$.
\end{lema}

The following result will be key for some estimates in our main results. The proof can be found in \cite[p. 538]{Berra-Carena-Pradolini(M)}.

\begin{lema}\label{lema: lema fundamental}
	Let $u\in A_1$ and $v$ such that $uv\in A_\infty$. Then there exists a positive constant $C$ such that for every cube $Q$
	\[\frac{uv(Q)}{v(Q)}\leq C\inf_Q u.\] 
\end{lema}

An important property of  Muckenhoupt weights is the reverse H\"{o}lder condition. Given $1<s<\infty$ we say that $w\in \mathrm{RH}_s$ if there exists a positive constant $C$
 such that for every cube~$Q$
\begin{equation*}
\left(\frac{1}{|Q|}\int_Q w^s(x)\,dx\right)^{1/s}\leq
\frac{C}{|Q|}\int_Q w(x)\,dx.
\end{equation*}
We say $w$ belongs to RH$_\infty$ if
there exists a positive constant $C$ such
that
\begin{equation*}
\sup_Q w\leq\frac{C}{|Q|}\int_Q w,
\end{equation*}
for every $Q\subset \R^n$. 
We denote by $[w]_{\textrm{RH}_s}$ the smallest constant $C$ for
which the conditions above hold. Notice that
$\textrm{RH}_\infty\subseteq \textrm{RH}_s\subseteq \textrm{RH}_q$,
for every $1<q<s$.

We say that $\varphi:[0,\infty)\to [0,\infty]$ is a Young function if it is strictly increasing, convex,  $\varphi(0)=0$ and $\varphi(t)\to\infty$ when $t\to\infty$. Given a Young function $\varphi$
and a Muckenhoupt weight $w$, the generalized maximal
	operator $M_{\varphi,w}$ is defined by
\[M_{\varphi,w} f(x):=\sup_{Q\ni x}\|f\|_{\varphi,Q,w},\]
where $\|f\|_{\varphi,Q,w}$ denotes the weighted Luxemburg average over $Q$ defined by
\begin{equation}\label{eq: luxem norm}
\|f\|_{\varphi,Q,w}:=\inf\left\{\lambda>0:
\frac{1}{w(Q)}\int_Q\varphi\left(\frac{|f|}{\lambda}\right)w\leq
1\right\}.
\end{equation}
When $w=1$ we simply denote $\|f\|_{\varphi,Q,w}=\|f\|_{\varphi,Q}$. We use $\bar{\varphi}$ to denote the
complementary function \label{page: complementary} associated to $\varphi$, defined
for $t\geq 0$ by
\[\bar{\varphi}(t)=\sup\{ts-\varphi(s):s\geq 0\}.\]
It is well known that $\bar{\varphi}$ satisfies
\[t\leq \varphi^{-1}(t)\bar{\varphi}^{-1}(t)\leq 2t,\quad\forall t>0,\]
where $\varphi^{-1}$ denotes the generalized inverse of $\varphi$, defined by
\[\varphi^{-1}(t)=\inf\{s\geq 0: \varphi(s)\geq t\},\]
where we understand $\inf\emptyset=\infty$.

We say that $\varphi$ belongs to $B_p$, $1<p<\infty$, if there exists $c>0$ satisfying
\begin{equation}
\int_c^\infty \frac{\varphi(t)}{t^p}\,\frac{dt}{t}<\infty.
\end{equation}

These classes were introduced in \cite{Perez-95-Onsuf} and it was shown that $\varphi \in B_p$ if and only if $M_{\varphi}$ is bounded on $L^p(\mathbb R^n)$.
Observe that if $1<p\leq q<\infty$, then $B_p\subseteq B_q$.

The following lemma contains a useful fact about $M_\varphi$ and $A_1$ weights. 

\begin{lema}[\cite{Berra-Carena-Pradolini(M)}, Lemma 15]\label{lema: control puntual de MPhi u por u}
	Let $s>1$, $u\in A_1\cap \mathrm{RH}_s$ and $\varphi$ a Young function in $B_s$. Then there exists a positive constant $C$ such that $M_\varphi u(x)\leq Cu(x)$, for almost every $x$.
\end{lema}

We will be dealing with functions of the type $\Phi(t)=t(1+\log^+t)^\delta$, with $\delta>0$. It is well-known that
\begin{equation}\label{eq: inversa y complementaria de LlogL}
\Phi^{-1}(t) \approx \frac{t}{(1+\log^+t)^{\delta}} \quad\textrm{ and }\quad \bar{\Phi}(t)\approx\left(e^{t^{1/\delta}}-e\right)\mathcal{X}_{(1,\infty)}(t).
\end{equation}
When $\delta=1/r$, $r\geq 1$, we shall refer to the average $\|\cdot\|_{\bar{\Phi}, Q, w}$ as $\|\cdot\|_{\mathrm{exp}\,{L^r}, Q, w}$.

The generalized H\"{o}lder inequality
\begin{equation*}
\frac{1}{w(Q)}\int_Q |fg| w\lesssim  \|f\|_{\varphi,Q,w}\|g\|_{\bar{\varphi},Q,w}
\end{equation*}
holds for every Muckenhoupt weight $w$ and every Young function $\varphi$. More generally, if $\varphi$, $\phi$ and $\eta$ are Young functions such that the inequality 
\begin{equation}\label{eq: holder generalizada}
\phi^{-1}(t)\eta^{-1}(t)\leq C \varphi^{-1}(t)
\end{equation}
holds for every $t \geq t_0 > 0$, then
\begin{equation*}
\|fg\|_{\varphi,Q,w}\leq C \|f\|_{\phi,Q,w}\|g\|_{\eta,Q,w}.
\end{equation*}
We can also replace the cube $Q$ above with any bounded measurable set $E$.

The next lemma is a variant of Lemma~11 in \cite{Berra-Carena-Pradolini(M)}, where it was shown for the case $r=1$. Its proof can be achieved with minor modifications and we will omit it.

\begin{lema}\label{lema: control de norma exp L^r respecto de w por norma exp L^r}
	Let $r\geq1$ and $w\in \mathrm{RH}_s$, for some $s>1$. Then there exists $C>0$ such that
	\[\|f\|_{\mathrm{exp}\,L^r,Q,w}\leq C\|f\|_{\mathrm{exp}\,L^r,Q}\]
	for every cube $Q$.
\end{lema}

The following fact is a well-known property satisfied for $\mathrm{BMO}$ symbols. A proof for the $\mathrm{Osc}_{\mathrm{exp}\, L^r}$ case can be found in \cite{P-T}.   
\begin{lema}\label{lema: promedios sobre cubos dilatados de b acotada por norma osc}
	Let $r\geq 1$ and $b\in \mathrm{Osc}_{\mathrm{exp}\, L^r}$. For every $k\in \mathbb{N}$ and every cube $Q$ we have that
	\[|b_{Q}-b_{2^k Q}|\leq C k \|b\|_{\mathrm{Osc}_{\mathrm{exp}\, L^r}}.\]
\end{lema}

The next is a purely technical result, which will be useful in some of our estimates. The proof can be performed by induction and we shall omit it.

\begin{lema}\label{lema: producto de sumas}
	Let $a_i$ and $b_i$ real numbers for $1\leq i\leq m$. Then
	\[\prod_{i=1}^m(a_i+b_i)=\sum_{\sigma\in S_m} \prod_{i=1}^ma_i^{\sigma_i}b_i^{\bar\sigma_i}.\]
\end{lema}
 Let $T$ be as in \eqref{eq: representacion integral del T}. Recall that $T$ is a CZO if $K$ is a standard kernel, which means that
$K:\mathbb R^n\backslash\{0\}\to\mathbb C$  satisfies a size condition given by 
\[|K(x)|\lesssim \frac{1}{|x|^n},\]
and the smoothness conditions, usually called Lipschitz conditions,
\begin{equation}\label{eq:prop del nucleo}
	|K(x-y)-K(x-z)|\lesssim \frac{|x-z|}{|x-y|^{n+1}},\quad \textrm{ if } |x-y|>2|y-z|,
\end{equation}
We write $f(t)\lesssim g(t)$ when there exists a positive constant $C$ such that $f(t)\leq Cg(t)$ for every $t$ in the domain.
Throughout the paper, the constant $C$ may change on each occurrence.

We shall also deal with kernels with less regularity properties than \eqref{eq:prop del nucleo}. Given a Young function $\varphi$, we denote
	\[\norm{f}_{\varphi, |x|\sim s}=\norm{f\mathcal{X}_{\{|x|\sim s\}}}_{\varphi, B(0,2s)},\]
	where $|x|\sim s$ means $s<|x|\leq 2s$.
	
	We say that $K$ satisfies the $L^{\varphi}-$H\"{o}rmander condition and we denote it by $K\in H_{\varphi}$ if there exist constants $c\geq 1$ and $C_{\varphi}>0$ such that, for every $y\in \R^n$ and $R>c|y|$
	\begin{equation}\label{eq: condicion_hormander_phi}
		\sum_{k=1}^{\infty}(2^kR)^n\norm{K(\cdot-y)-K(\cdot)}_{\varphi,|x|\sim 2^kR}\leq C_\varphi.
	\end{equation}
	We also say that $K\in H_{\varphi,m}$ for $m\in\N$ if there exist constants $c\geq 1$ and $C_{\varphi,m}>0$ such that the inequality
	\begin{equation}\label{eq: condicion_hormander_phi_m}
		\sum_{k=1}^{\infty}(2^kR)^nk^m\norm{K(\cdot-y)-K(\cdot)}_{\varphi,|x|\sim 2^kR}\leq C_{\varphi,m}.
	\end{equation}
	holds for every $y\in \mathbb{R}^n$ and $R>c|y|$. It is immediate from the definition that $H_{\varphi,m}\subset H_{\varphi,i}$ for every $0\leq i\leq m$.

\section{Some useful facts about $T_{\mathbf{b}}$}\label{seccion: propiedades de Tb}

We devote this section to prove some technical results concerning the representation of commutators with multilinear symbols. Although the most part of them do not rely on specific properties of $T$, we shall assume that $T$ is a linear operator.
We introduce some useful notation. Given $m\in\mathbb{N}$, we denote with $S_m$ the set\label{pag: notacion para Tb}
	\[S_m=\{0,1\}^m=\{\sigma \in \mathbb{R}^m: \sigma=(\sigma_1,\sigma_2,\dots,\sigma_m), \textrm{ where } \sigma_i=0 \textrm{ or } \sigma_i=1\}.\]
	Notice that $S_m$ has exactly $2^m$ elements. We shall write this as $\#S_m=2^m$. For $\sigma\in S_m$ we define
	$|\sigma|=\sum_{i=1}^m\sigma_i$ and $\bar{\sigma}=(\bar{\sigma}_1,\bar{\sigma}_2,\dots,\bar{\sigma}_m)$,
	where
	\[\bar{\sigma}_i=\left\{\begin{array}{ccl}
		0&\textrm{ if }& \sigma_i=1,\\
		1&\textrm{ if }& \sigma_i=0.
	\end{array}
	\right.\]
	
	If $\mathbf{b}=(b_1,\dots,b_m)$ and $\sigma\in S_m$ we denote by $\mathbf{b}_\sigma$ the symbol with $|\sigma|$ components corresponding to those indices $i$ for which $\sigma_i=1$. For example, for $m=4$, $\mathbf{b}=(b_1,b_2,b_3,b_4)$ and $\sigma=(1,0,1,0)$, we have $\mathbf{b}_\sigma=(b_1,b_3)$. When we need to emphasize that $\mathbf{b}_\sigma$ has length $k$ we will write $\mathbf{b}_\sigma^{k}$. Also, we shall use $\overline{\mathbf{b}}_{i}$ to denote the multilinear symbol containing all the components of $\mathbf{b}$ except $b_i$, this is  
	\[\overline{\mathbf{b}}_{i}=(b_1,\dots,b_{i-1},b_{i+1},\dots b_m).\]

\begin{propo}\label{propo: representacion de Tb con oscilaciones medias de bi}
	Let $\mathbf{b}=(b_1,b_2,\dots,b_m)$ and $\lambda_i$ be fixed constants for $1\leq i\leq m$. Then
	\begin{equation}\label{eq: representacion de Tb con oscilaciones medias de bi}
	T_{\mathbf{b}}f(x)=\sum_{\sigma\in S_m}(-1)^{m-|\sigma|}\left(\prod_{i=1}^m \left[b_i(x)-\lambda_i\right]^{\sigma_i}\right)T\left(\left(\prod_{i=1}^m \left[b_i-\lambda_i\right]^{\bar{\sigma}_i}\right)f\right)(x).
	\end{equation}
\end{propo}

\begin{proof}
	We proceed again by induction on $m$. If $m=1$ then $\mathbf{b}=(b_1)$. In this case, the right-hand side of \eqref{eq: representacion de Tb con oscilaciones medias de bi} becomes
	\[\left(b_1(x)-\lambda_1\right)Tf(x)-T\left(\left(b_1-\lambda_1\right)f\right)(x)=b_1(x)Tf(x)-T(b_1f)(x)=T_{\mathbf{b}}f(x),\]
	since $T$ is linear.
	
	Assume now that the representation holds for every multilinear symbol with $k$ components and
		let $\mathbf{b}=(b_1,b_2,\dots,b_{k+1})$. The right-hand side of \eqref{eq: representacion de Tb con oscilaciones medias de bi} can be split in two terms $A+B$, where
	\[A=\sum_{\sigma\in S_{k+1}: \sigma_1=1}(-1)^{k+1-|\sigma|}\left(b_1(x)-\lambda_1\right)\left(\prod_{i=2}^{k+1} \left[b_i(x)-\lambda_i\right]^{\sigma_i}\right)T\left(\left(\prod_{i=2}^{k+1} \left[b_i-\lambda_i\right]^{\bar{\sigma}_i}\right)f\right)(x)\]
	and
	\[B=\sum_{\sigma\in S_{k+1}: \sigma_1=0}(-1)^{k+1-|\sigma|}\left(\prod_{i=2}^{k+1} \left[b_i(x)-\lambda_i\right]^{\sigma_i}\right)T\left(\left(b_1-\lambda_1\right)\left(\prod_{i=2}^{k+1} \left[b_i-\lambda_i\right]^{\bar{\sigma}_i}\right)f\right)(x).\]
	Now, by using the inductive hypothesis, we have that
	\begin{align*}
	A&=\left(b_1(x)-\lambda_1\right)\left[\sum_{\theta\in S_k} (-1)^{k-|\theta|}\left(\prod_{i=2}^{k+1} \left[b_i(x)-\lambda_i\right]^{\theta_i}\right)T\left(\left(\prod_{i=2}^{k+1} \left[b_i-\lambda_i\right]^{\bar{\theta}_i}\right)f\right)(x)\right]\\
	&=(b_1(x)-\lambda_1)T_{\overline{\mathbf{b}}_{1}}f(x).
	\end{align*}
	On the other hand, since $T$ is linear 
	\begin{align*}
	B&=-\sum_{\theta\in S_k}(-1)^{k-|\theta|}\left(\prod_{i=2}^{k+1} \left[b_i(x)-\lambda_i\right]^{\theta_i}\right)T\left(\left(b_1-\lambda_1\right)\left(\prod_{i=2}^{k+1} \left[b_i-\lambda_i\right]^{\bar{\theta}_i}\right)f\right)(x)\\
	&=-\sum_{\theta\in S_k}(-1)^{k-|\theta|}\left(\prod_{i=2}^{k+1} \left[b_i(x)-\lambda_i\right]^{\theta_i}\right)T\left(b_1\left(\prod_{i=2}^{k+1} \left[b_i-\lambda_i\right]^{\bar{\theta}_i}\right)f\right)(x)+\\
	&\quad + \lambda_1\sum_{\theta\in S_k}(-1)^{k-|\theta|}\left(\prod_{i=2}^{k+1} \left[b_i(x)-\lambda_i\right]^{\theta_i}\right)T\left(\left(\prod_{i=2}^{k+1} \left[b_i-\lambda_i\right]^{\bar{\theta}_i}\right)f\right)(x)\\
	&=-T_{\overline{\mathbf{b}}_{1}}(b_1f)(x)+\lambda_1T_{\overline{\mathbf{b}}_{1}}f(x).
	\end{align*}
	Finally, we obtain that 
	\[A+B=b_1(x)T_{\overline{\mathbf{b}}_1}f(x)-T_{\overline{\mathbf{b}}_1}(b_1f)(x)=\left[b_1, T_{\overline{\mathbf{b}}_{1}}f\right](x)=T_{\mathbf{b}}f(x),\]
	which completes the proof.\qedhere
\end{proof}

\begin{obs}
	If we take $\lambda_i=0$ for every $i$ in the result above, we obtain
	\begin{equation}\label{eq: representacion de Tb con b multilineal}
	T_{\mathbf{b}}f(x)=\sum_{\sigma\in S_m}(-1)^{m-|\sigma|}\left(\prod_{i=1}^m b_i^{\sigma_i}(x)\right)T\left(\left(\prod_{i=1}^m b_i^{\bar{\sigma}_i}\right)f\right)(x).
	\end{equation}
\end{obs}

\begin{coro}\label{coro: definicion recursiva de Tb independiente de la componente con que se conmuta}
	Given $1\leq \ell\leq m$ and $\mathbf{b}=(b_1,b_2,\dots,b_m)$ we have that
	\[T_{\mathbf{b}}f(x)=\left[b_\ell,T_{\overline{\mathbf{b}}_{\ell}}f\right](x).\]
\end{coro}

\begin{proof}
	By virtue of \eqref{eq: representacion de Tb con b multilineal} we have that
	\[T_{\mathbf{b}}f(x)=\sum_{\sigma\in S_m}(-1)^{m-|\sigma|}\left(\prod_{i=1}^m b_i^{\sigma_i}(x)\right)T\left(\left(\prod_{i=1}^m b_i^{\bar{\sigma}_i}\right)f\right)(x)=\sum_{\sigma\in S_m: \sigma_\ell=0}+\sum_{\sigma\in S_m: \sigma_\ell=1}.\]
	By writing these sums over elements of $S_{m-1}$ we get
	\[T_{\mathbf{b}}f(x)=b_\ell(x)T_{\overline{\mathbf{b}}_{\ell}}f(x)-T_{\overline{\mathbf{b}}_{\ell}}(b_\ell f)(x),\]
	as desired.
\end{proof}

We shall deal with linear operators $T$ such that $Tf(x)$ has an integral representation for adequate values of $x$. The following proposition states that its commutators have the same property. The proof is straightforward and we shall omit it.

\begin{propo}\label{propo: linealidad y rep integral de Tb}
	Let $T$ be linear, $m\in \mathbb{N}$ and $\mathbf{b}=(b_1,\dots,b_m)$ be a  multilinear symbol. Then 
	\begin{enumerate}[label=(\alph*)]
		\item\label{item:lineal} $T_\mathbf{b}$ is a linear operator;
		\item\label{item:representacion} if $T$ has the representation
		\begin{equation}\label{eq: representacion integral de T}
		Tf(x)=\int_{\mathbb{R}^n} K(x-y)f(y)\,dy
		\end{equation}
		for every $x\not\in \mathrm{supp}(f)$, then
		\begin{equation}\label{eq: representacion integral de Tb}
		T_\mathbf{b}f(x)=\int_{\mathbb{R}^n}\prod_{i=1}^m\left(b_i(x)-b_i(y)\right)K(x-y)f(y)\,dy
		\end{equation}
		for these $x$.
	\end{enumerate}
\end{propo}

The following result states a representation for $T_\mathbf{b}$ by means of lower order commutators.

\begin{propo}\label{propo: Tb en termino de conmutadores con simbolos de menor orden}
	Let $T$ be an operator with the representation given by \eqref{eq: representacion integral de T} for $x\not\in\mathrm{supp}(f)$ and $\mathbf{b}=(b_1,\dots,b_m)$. If $\lambda_i$ are fixed constants for every $1\leq i\leq m$ and $x\not\in \mathrm{supp}(f)$, then
	\begin{align*}
	T_{\mathbf b} f(x)&=\left(\prod_{i=1}^m(b_i(x)-\lambda_i)\right)Tf(x)-T\left(\left(\prod_{i=1}^m(b_i-\lambda_i)\right)f\right)(x)\\
	&\quad -\sum_{\sigma\in S_m,\,0<|\sigma|<m }T_{\mathbf b_\sigma}\left(\left(\prod_{i=1}^m(b_i-\lambda_i)^{\bar\sigma_i}\right)f\right)(x).
	\end{align*}
\end{propo}

\begin{proof}
	By induction on $m$, if $\mathbf{b}=(b_1)$ by applying Proposition~\ref{propo: representacion de Tb con oscilaciones medias de bi} we have that
	\begin{align*}
	T_{\mathbf b} f(x)&=(b_1(x)-\lambda_1)Tf(x)-T((b_1-\lambda_1)f)(x),
	\end{align*}
	which implies the thesis since the set
	\[\{\sigma \in S_1:  0<|\sigma|<1\}\]
	is empty. Assume now that the equality holds for every multilinear symbol with $k$ components and let $\mathbf b=(b_1,\dots,b_k,b_{k+1})$. From Corollary~\ref{coro: definicion recursiva de Tb independiente de la componente con que se conmuta} and Proposition~\ref{propo: linealidad y rep integral de Tb} we can write
	\[T_{\mathbf{b}}f(x)=[b_{k+1},T_{\overline{\mathbf{b}}_{k+1}}f](x)=(b_{k+1}(x)-\lambda_{k+1})T_{\overline{\mathbf{b}}_{k+1}}f(x)-T_{\overline{\mathbf{b}}_{k+1}}((b_{k+1}-\lambda_{k+1})f)(x),\]
	where $\overline{\mathbf{b}}_{k+1}=(b_1,\dots,b_k)$.
	By the inductive hypothesis we write the first term as 		 
	\begin{align*}
	(b_{k+1}(x)-\lambda_{k+1})T_{\overline{\mathbf{b}}_{k+1}}f(x)&=\left(\prod_{i=1}^{k+1}(b_i(x)-\lambda_i)\right)Tf(x)\\
	&\quad -(b_{k+1}(x)-\lambda_{k+1})T\left(\left(\prod_{i=1}^{k}(b_i-\lambda_i)\right)f\right)(x)\\
	&\quad-\sum_{\theta \in S_k,\,0<|\theta|<k}(b_{k+1}(x)-\lambda_{k+1})T_{\mathbf b_{\theta}}\left(\left(\prod_{i=1}^{k}(b_i-\lambda_i)^{\bar\theta_{i}}\right)f\right)(x)\\
	&=A+B+C.
	\end{align*}

	We shall use Proposition~\ref{propo: linealidad y rep integral de Tb} to estimate some of the previous terms. We have
	\begin{align*}
	B&=-\int_{\mathbb{R}^n}\left(b_{k+1}(x)-b_{k+1}(y)+b_{k+1}(y)-\lambda_{k+1}\right)\left(\prod_{i=1}^k (b_i(y)-\lambda_i)\right)K(x-y)f(y)\,dy\\
	&=-T_{b_{k+1}}\left(\left(\prod_{i=1}^{k}(b_i-\lambda_i)\right)f\right)(x)-T\left(\left(\prod_{i=1}^{k+1}(b_i-\lambda_i)\right)f\right)(x).
	\end{align*}
	On the other hand, if we set
	\[G_{k+1}=\left\{\sigma\in S_{k+1}: \sigma_{k+1}=1 \textrm{ and } 0<\sum_{i=1}^k \sigma_i<k\right\}\]
	and
	\[H_{k+1}=\left\{\sigma\in S_{k+1}: \sigma_{k+1}=0 \textrm{ and } 0<\sum_{i=1}^k \sigma_i<k\right\},\]
	we can write
	\begin{align*}
	C&=-\sum_{\theta \in S_k,\, 0<|\theta|<k}\int_{\mathbb{R}^n}(b_{k+1}(x)-\lambda_{k+1})\left(\prod_{i=1}^k(b_i(x)-b_i(y))^{\theta_i}\right)\left(\prod_{i=1}^k(b_i(y)-\lambda_i)^{\bar\theta_i}\right)K(x-y)f(y)\,dy\\
	&=-\sum_{\theta \in S_k,\, 0<|\theta|<k}\int_{\mathbb{R}^n}(b_{k+1}(x)-b_{k+1}(y))\left(\prod_{i=1}^k(b_i(x)-b_i(y))^{\theta_i}\right)\left(\prod_{i=1}^k(b_i(y)-\lambda_i)^{\bar\theta_i}\right)K(x-y)f(y)\,dy\\
	&\quad -\sum_{\theta \in S_k,\, 0<|\theta|<k}\int_{\mathbb{R}^n}(b_{k+1}(y)-\lambda_{k+1})\left(\prod_{i=1}^k(b_i(x)-b_i(y))^{\theta_i}\right)\left(\prod_{i=1}^k(b_i(y)-\lambda_i)^{\bar\theta_i}\right)K(x-y)f(y)\,dy\\
	&=-\sum_{\sigma\in G_{k+1}} T_{\mathbf{b}_\sigma}\left(\prod_{i=1}^{k+1}(b_i-\lambda_i)^{\bar \sigma_i}f\right)(x)-\sum_{\sigma\in H_{k+1}} T_{\mathbf{b}_\sigma}\left(\prod_{i=1}^{k+1}(b_i-\lambda_i)^{\bar \sigma_i}f\right)(x).
	\end{align*}
	This yields the thesis since 
	\[G_{k+1}\cup H_{k+1}\cup \{(0,0,\dots,0,1),(1,1,\dots,1,0)\}=\{\sigma\in S_{k+1}: 0<|\sigma|<k+1\}.\qedhere\]
	\end{proof}

\begin{propo}\label{propo: Tb en terminos de conmutadores con simbolos de menor orden aplicado a f}
	If $\mathbf{b}=(b_1,\dots,b_m)$ and $\lambda_i$ are fixed constants for every $1\leq i\leq m$, then
	\[T_{\mathbf{b}}f=(-1)^mT\left(\left(\prod_{i=1}^m(b_i-\lambda_i)\right)f\right)+\sum_{\sigma\in S_m,\, |\sigma|<m}(-1)^{m-1-|\sigma|}\left(\prod_{i=1}^m (b_i-\lambda_i)^{\bar\sigma_i}\right)T_{\mathbf{b}_\sigma} f.\]
\end{propo}

\begin{proof}
	We proceed by induction. Recall that $T_{\mathbf{b}_\sigma}=T$ when each component of $\sigma$ is zero, so that the result trivially holds when $m=1$. Now assume that it is true for every multilinear symbol with $k$ components and let us prove for $\mathbf{b}=(b_1,b_2,\dots,b_{k+1})$. For simplicity, we will denote  $\tilde{\mathbf{b}}=\overline{\mathbf{b}}_{k+1}=(b_1,b_2,\dots,b_k)$. Then, from Corollary~\ref{coro: definicion recursiva de Tb independiente de la componente con que se conmuta} we have that
	\begin{align*}
	T_{\mathbf{b}}f(x)&=[b_{k+1},T_{\tilde{\mathbf{b}}}]f(x)\\
	&=(b_{k+1}(x)-\lambda_{k+1})T_{\tilde{\mathbf{b}}}f(x)-T_{\tilde{\mathbf{b}}}((b_{k+1}-\lambda_{k+1})f)(x)\\
	&=(b_{k+1}(x)-\lambda_{k+1})T_{\tilde{\mathbf{b}}}f(x)-(-1)^kT\left(\left(\prod_{i=1}^{k+1}(b_i-\lambda_i)\right)f\right)(x)+\\
	&\quad+\sum_{\theta\in S_k,\, |\theta|<k}(-1)^{k-|\theta|}\left(\prod_{i=1}^k(b_i-\lambda_i)^{\bar\theta_i}\right)T_{\mathbf{b}_\theta}((b_{k+1}-\lambda_{k+1})f)(x).
	\end{align*}
	By defining the sets
	\[H_k=\left\{\sigma\in S_k: \sigma_k=1 \textrm{ and } \prod_{i=1}^{k-1} \sigma_i=0\right\}\quad\textrm{ and }\quad G_k=\left\{\sigma\in S_k: \sigma_k=0 \textrm{ and } \prod_{i=1}^{k-1} \sigma_i=0\right\}\]
	we get
	\[\{\theta\in S_k, |\theta|<k\}=H_k\cup G_k\cup\left\{\eta^{k}\right\},\]
	where $\eta^{k}=(1,1,1,\dots,1,0)\in S_k$.
	Since $\# (H_k\cup G_k)=\# H_{k+1}$, for every $\sigma\in H_{k+1}$ we can associate a vector $\theta=\theta(\sigma)\in S_k$, with at least one of its components equal to 0 such that $\sigma_i=\theta_i$ for each $1\leq i\leq k$. By using the relation 
	\begin{align*}
	T_{\mathbf{b}_\sigma}f(x)&=[b_{k+1},T_{\mathbf{b}_\theta}]f(x)\\
	&=(b_{k+1}-\lambda_{k+1})T_{\mathbf{b}_\theta}f(x)-T_{\mathbf{b}_\theta}((b_{k+1}-\lambda_{k+1})f)(x).
	\end{align*}
	we can write the sum
	\[\sum_{\theta\in S_k,\, |\theta|<k}(-1)^{k-|\theta|}\left(\prod_{i=1}^k(b_i-\lambda_i)^{\bar\theta_i}\right)T_{\mathbf{b}_\theta}((b_{k+1}-\lambda_{k+1})f)(x)\]
	as
	\[\sum_{\sigma\in H_{k+1}}(-1)^{k-|\sigma|+1}\left(\prod_{i=1}^k(b_i-\lambda_i)^{\bar\sigma_i}\right)\left((b_{k+1}-\lambda_{k+1})T_{\mathbf{b}_\theta}f(x)-T_{\mathbf{b}_\sigma}f(x)\right)=I+II,\]
	where
	\[I=\sum_{\sigma\in G_{k+1}}(-1)^{k-|\sigma|}\left(\prod_{i=1}^{k+1}(b_i-\lambda_i)^{\bar\sigma_i}\right)T_{\mathbf{b}_\sigma}f(x)\]
	and
	\[II=\sum_{\sigma\in H_{k+1}}(-1)^{k-|\sigma|}\left(\prod_{i=1}^k(b_i-\lambda_i)^{\bar\sigma_i}\right)\left((b_{k+1}-\lambda_{k+1})T_{\mathbf{b}_\sigma}f(x)\right).\]
	Finally, the equality
	\[\{\sigma\in S_{k+1},|\sigma|<k+1\}=H_{k+1}\cup G_{k+1}\cup \left\{\eta^{k+1}\right\},\]
	where $\eta^{k+1}=(1,1,1,\dots,1,0)\in S_{k+1}$, implies that the sum
	\[(b_{k+1}(x)-\lambda_{k+1})T_{\tilde{\mathbf{b}}}f(x)+\sum_{\theta\in S_k,\, |\theta|<k}(-1)^{k-|\theta|}\left(\prod_{i=1}^k(b_i-\lambda_i)^{\bar\theta_i}\right)T_{\mathbf{b}_\theta}((b_{k+1}-\lambda_{k+1})f)(x)\]
	is equal to
	\[\sum_{\sigma\in S_{k+1},\, |\sigma|<k+1}(-1)^{k-|\sigma|}\left(\prod_{i=1}^{k+1}(b_i-\lambda_i)^{\bar\sigma_i}\right)T_{\mathbf{b}_\sigma}f(x).\]
	This concludes the proof.
\end{proof}

\section{Mixed inequalities for commutators of CZO}\label{seccion: caso CZO}

We shall use a mixed weak type estimate for $T$ given below. This result was set and proved in \cite{CruzUribe-Martell-Perez}. Also, the proof can be performed in an alternative way by following the ideas involved in \eqref{eq: desigualdad mixta para Tbm OCZ b en BMO} in \cite{Berra-Carena-Pradolini(M)}.

\begin{teo}[\cite{CruzUribe-Martell-Perez}]\label{teo: desigualdad mixta para T}
	Let $T$ be a Calder\'on-Zygmund operator. If $u,v\in A_1$ or  $u\in A_1$ and $v\in A_\infty(u)$, then there exists a positive constant $C$ such that the inequality
	\[uv\left(\left\{x\in \mathbb{R}^n: \frac{|T(fv)(x)|}{v(x)}>t\right\}\right)\leq \frac{C}{t}\int_{\mathbb{R}^n} |f(x)|u(x)v(x)\,dx,\]
	holds for every $t>0$ and every bounded function $f$ with compact support.
\end{teo}

\begin{proof}[Proof of Theorem~\ref{teo: desigualdad mixta para Tb caso CZO}]
We shall proceed by induction on $m$. The case corresponding to $m=1$ can be achieved by following similar lines as in the proof of \eqref{eq: desigualdad mixta para Tbm OCZ b en BMO} by substituting $\mathrm{BMO}$ condition by $\mathrm{Osc}_{\mathrm{exp}\, L^r}$, with $r\geq 1$. For more clarity, we shall provide the details for the case $m=2$. 

Let us assume, without loss of generality, that $f$ is nonnegative, bounded and has compact support. From \eqref{eq: representacion de Tb con b multilineal} we obtain that
	\[T_\mathbf{b}\left(\frac{f}{\|b_1\|_{\mathrm{Osc}_{\mathrm{exp}\, L^{r_1}}}\|b_2\|_{\mathrm{Osc}_{\mathrm{exp}\, L^{r_2}}}}\right)(x)=T_{\tilde{\mathbf{b}}}f(x),\]
	where $\tilde{\mathbf{b}}=(b_1/\|b_1\|_{\mathrm{Osc}_{\mathrm{exp}\, L^{r_1}}},b_2/\|b_2\|_{\mathrm{Osc}_{\mathrm{exp}\, L^{r_2}}})$. Therefore, it will be enough to achieve the estimate for the case in which $\|b_i\|_{\mathrm{Osc}_{\mathrm{exp}\, L^{r_i}}}=1$ for each $i$.

	The hypothesis on $v$ ensures that it belongs to $A_\infty$. We fix $t>0$ and perform the Calder\'on-Zygmund decomposition of $f$ at level $t$, with respect to the measure given by $d\mu(x)=v(x)\,dx$. This yields a disjoint collection of dyadic cubes $\{Q_j\}_{j=1}^\infty$ that verify
	\[t<\frac{1}{v(Q_j)}\int_{Q_j}fv\leq C_n t,\]
	for every $j$. We also split $f=g+h$, where
	\[g(x)=\left\{\begin{array}{ccl}
	f(x)&\textrm{ if }& x\in \mathbb{R}^n\backslash \Omega;\\
	f_{Q_j}^v&\textrm{ if }& x\in Q_j,
	\end{array}
	\right.\]
	where $f_{Q_j}^v$ denotes the average of $f$ over $Q_j$ with respect to the measure $\mu$ and $\Omega=\cup_j Q_j$. Also $h(x)=\sum_j h_j(x)$, where
	\[h_j(x)=(f(x)-f_{Q_j}^v)\mathcal{X}_{Q_j}(x).\]
	That is, every function $h_j$ is supported in $Q_j$ and  
	\begin{equation}\label{eq: integral cero de hj}
	\int_{Q_j}h_jv=0
	\end{equation}
	for every $j$. If we set $Q_j^*=3Q_j$ and $\Omega^*=\bigcup_j Q_j^*$, then we obtain
	\begin{align*}
	uv\left(\left\{x: \frac{|T_\mathbf{b}(fv)(x)|}{v(x)}>t\right\}\right)&= uv\left(\left\{x: \frac{|T_\mathbf{b}(gv)(x)|}{v(x)}>\frac{t}{2}\right\}\right)+uv\left(\left\{x: \frac{|T_\mathbf{b}(hv)(x)|}{v(x)}>\frac{t}{2}\right\}\right)\\
	&\leq uv\left(\left\{x: \frac{|T_\mathbf{b}(gv)(x)|}{v(x)}>\frac{t}{2}\right\}\right)+uv(\Omega^{*})\\
	&\quad +
	uv\left(\left\{x\in \mathbb{R}^n\backslash{\Omega^*}: \frac{|T_\mathbf{b}(hv)(x)|}{v(x)}>\frac{t}{2}\right\}\right)\\
	&=I+II+III.
	\end{align*}
	We begin with the estimate of $I$. Since $v\in A_\infty(u)$, there exists $p>1$ such that $v\in A_p(u)$. This implies that $v^{1-p'}\in A_{p'}(u)$ and by Lemma~\ref{lema: u en A1 y v en Ap(u) implica uv en Ap} we get $uv^{1-p'}\in A_{p'}$. By applying Tchebychev inequality and the strong type estimate for $T_{\mathbf{b}}$ proved in \cite{P-T} we have
	\label{pag: estimacion de I}
	\begin{align*}
	I&\leq \frac{1}{t^{p'}}\int_{\mathbb{R}^n}|T_{\mathbf{b}}(gv)|^{p'}uv^{1-p'}\\
	&\leq \frac{C}{t^{p'}}\int_{\mathbb{R}^n}g^{p'}uv\\
	&\leq \frac{C}{t}\left(\int_{\mathbb{R}^n\backslash \Omega}fuv+\sum_j \int_{Q_j}f_{Q_j}^vuv\right)\\
	&\leq \frac{C}{t}\left(\int_{\mathbb{R}^n\backslash \Omega}fuv+\sum_j \frac{uv(Q_j)}{v(Q_j)}\int_{Q_j}fv\right)\\
	&\leq  \frac{C}{t}\left(\int_{\mathbb{R}^n\backslash \Omega}fuv+C\sum_j \int_{Q_j}fuv\right)\\
	&\leq \frac{C}{t}\int_{\mathbb{R}^n}fuv
	\end{align*}
	by virtue of Lemma~\ref{lema: lema fundamental}. We turn now to the estimate of $II$. Since $uv\in A_\infty$, it is a doubling weight. Then
	\begin{align*}
	II=uv(\Omega^*)
	&=\leq C\sum_j \frac{uv(Q_j)}{v(Q_j)}\frac{1}{t}\int_{Q_j}fv\\
	&\leq C\sum_j \frac{1}{t}\int_{Q_j}fuv\leq \frac{C}{t}\int_{\mathbb{R}^n}fuv.	
	\end{align*}	
\label{pag: estimacion de II}
	For $III$, we fix constants $\lambda_{i,j}$ for $i=1,2$ and every $j$ that will be chosen later. By applying Proposition~\ref{propo: representacion de Tb con oscilaciones medias de bi} we write
	\begin{align*}
	T_{\mathbf b}(hv)(x)&=\sum_j T_{\mathbf b}(h_jv)(x)\\
	&=\sum_j \left(T((b_1-\lambda_{1,j})(b_2-\lambda_{2,j})h_jv)-(b_2(x)-\lambda_{2,j})T((b_1-\lambda_{1,j})h_jv)(x)\right.\\
	&\quad \left.- (b_1(x)-\lambda_{1,j})T((b_2-\lambda_{2,j})h_jv)(x)+(b_1(x)-\lambda_{1,j})(b_2(x)-\lambda_{2,j})T(h_jv)(x)\right).
	\end{align*}
    If $x\in \mathbb{R}^n\backslash \Omega^*$ we can use the integral representation of $T$ to write
	\begin{align*}
	(b_2(x)-\lambda_{2,j})T((b_1-\lambda_{1,j})h_jv)(x)&= \sum_j \int_{\mathbb{R}^n\backslash \Omega^*}(b_2(x)-\lambda_{2,j})(b_1(y)-\lambda_{1,j})h_j(y)K(x-y)v(y)\,dy\\
	&=\sum_j \int_{\mathbb{R}^n\backslash \Omega^*}(b_2(x)-b_2(y))(b_1(y)-\lambda_{1,j})h_j(y)K(x-y)v(y)\,dy\\
	&\quad +\sum_j \int_{\mathbb{R}^n\backslash \Omega^*}(b_2(y)-\lambda_{2,j})(b_1(y)-\lambda_{1,j})h_j(y)K(x-y)v(y)\,dy\\
	&= \sum_j T_{b_2}((b_1-\lambda_{1,j})h_jv)(x)+\\
	&\quad +\sum_jT((b_1-\lambda_{1,j})(b_2-\lambda_{2,j})h_jv)(x).
	\end{align*}
	Similarly,
	\[(b_1-\lambda_{1,j})T((b_2-\lambda_{2,j})h_jv)=\sum_j T_{b_1}((b_2-\lambda_{2,j})h_jv)+\sum_jT((b_1-\lambda_{1,j})(b_2-\lambda_{2,j})h_jv).\]
	Therefore, if $x\in \mathbb{R}^n\backslash \Omega^*$, we have
	\begin{align*}
	T_{\mathbf b} (hv)(x)&=-\sum_j \left(T((b_1-\lambda_{1,j})(b_2-\lambda_{2,j})h_jv)+(b_1(x)-\lambda_{1,j})(b_2(x)-\lambda_{2,j})T(h_jv)(x)\right)\\
	&\quad -\sum_j T_{b_1}((b_2-\lambda_{2,j})h_jv)(x)-\sum_j T_{b_2}((b_1-\lambda_{1,j})h_jv)(x),
	\end{align*}
	which gives
	\begin{align*}
	III&\leq	uv\left(\left\{x\in \mathbb{R}^n\backslash{\Omega^*}: \sum_j \frac{|T((b_1-\lambda_{1,j})(b_2-\lambda_{2,j})h_jv)(x)|}{v(x)}>\frac{t}{8}\right\}\right)\\
	&\quad + uv\left(\left\{x\in \mathbb{R}^n\backslash{\Omega^*}: \sum_j \frac{|(b_1(x)-\lambda_{1,j})(b_2(x)-\lambda_{2,j})T(h_jv)(x)|}{v(x)}>\frac{t}{8}\right\}\right)\\
	&\quad +	uv\left(\left\{x\in \mathbb{R}^n\backslash{\Omega^*}: \frac{|T_{b_1}(\sum_j(b_2-\lambda_{2,j})h_jv)(x)|}{v(x)}>\frac{t}{8}\right\}\right)\\
	&\quad +	uv\left(\left\{x\in \mathbb{R}^n\backslash{\Omega^*}: \frac{|T_{b_2}(\sum_j(b_1-\lambda_{1,j})h_jv)(x)|}{v(x)}>\frac{t}{8}\right\}\right)\\
	&=I_1+I_2+I_3+I_4.
	\end{align*}
	Let us first estimate $I_1$. Let us fix $\lambda_{i,j}=b_{i,j}$, where as usual 
	\[b_{i,j}=\frac{1}{|Q_j|}\int_{Q_j} b_i(x)\,dx \quad \textrm{ for }\quad i=1,2 \,\textrm{ and every }j.\]
	By Theorem~\ref{teo: desigualdad mixta para T} we have
	\begin{align*}
	I_1&\leq \frac{C}{t}\int_{\mathbb{R}^n}\sum_j |b_1(x)-b_{1,j}||b_2(x)-b_{2,j}||h_j(x)|u(x)v(x)\,dx\\
	&\leq\frac{C}{t}\sum_j \int_{Q_j} |b_1(x)-b_{1,j}||b_2(x)-b_{2,j}|f(x)u(x)v(x)\,dx\\
	&\quad +\frac{C}{t}\sum_j\left(\int_{Q_j} |b_1(x)-b_{1,j}||b_2(x)-b_{2,j}|u(x)v(x)\,dx\right)\left(\frac{1}{v(Q_j)}\int_{Q_j}f(y)v(y)\,dy\right)\\
	&=A+B.
	\end{align*}
	We start with $A$. Let $\varphi_i(t)=e^{t^{r_i}}-1$, for $i=1,2$. Observe that
	\[\varphi_1^{-1}(t)\varphi_2^{-1}(t)\Phi^{-1}(t)\lesssim (\log t)^{1/r_1+1/r_2}\frac{t}{(\log t)^{1/r}}=t,\]
	 if $t\geq e$.	It is well-known that for a Young function $\phi$ and an $A_\infty$ weight $w$ we have
	 \begin{equation}\label{eq: equivalencia de norma luxemburgo con integral}
	 \|f\|_{\phi,E,w}\approx \inf_{\tau>0}\left\{\tau+\frac{\tau}{w(E)}\int_E \phi\left(\frac{|f(x)|}{\tau}\right)w(x)\,dx\right\},
	 \end{equation}
	 for every measurable set $E$ with $0<|E|<\infty$ (see, for example, \cite{raoren} or \cite{KR} for the case $w=1$; the proof for $w\in A_\infty$ can be achieved by adapting the argument).
	 
	  Since $uv\in A_\infty$ we apply generalized H\"{o}lder inequality with the functions $\varphi_1$, $\varphi_2$ and $\Phi$, with respect to the measure $d\nu(x)=u(x)v(x)\,dx$. We also combine Lemmas~\ref{lema: control de norma exp L^r respecto de w por norma exp L^r} and~\ref{lema: lema fundamental} with \eqref{eq: equivalencia de norma luxemburgo con integral} to get
	 \begin{align*}
	 A&\leq \frac{C}{t}\sum_j (uv)(Q_j)\|b_1-b_{1,j}\|_{\mathrm{exp}\,L^{r_1},Q_j,uv}\|b_2-b_{2,j}\|_{\mathrm{exp}\,L^{r_2},Q_j,uv}\|f\|_{\Phi,Q_j,uv}\\
	 &\leq \frac{C}{t}\sum_j (uv)(Q_j)\|b_1-b_{1,j}\|_{\mathrm{exp}\,L^{r_1},Q_j}\|b_2-b_{2,j}\|_{\mathrm{exp}\,L^{r_2},Q_j}\|f\|_{\Phi,Q_j,uv}\\
	 &\leq \frac{C}{t}\sum_j (uv)(Q_j)\left(t+\frac{t}{uv(Q_j)}\int_{Q_j}\Phi\left(\frac{f(x)}{t}\right)u(x)v(x)\,dx\right)\\
	 &\leq \frac{C}{t}\sum_j \frac{(uv)(Q_j)}{v(Q_j)}\int_{Q_j}f(x)v(x)\,dx+C\sum_j\int_{Q_j}\Phi\left(\frac{f(x)}{t}\right)u(x)v(x)\,dx\\
	 &\leq C \sum_j\int_{Q_j}\Phi\left(\frac{f(x)}{t}\right)u(x)v(x)\,dx\\
	 &\leq C \int_{\mathbb{R}^n}\Phi\left(\frac{f(x)}{t}\right)u(x)v(x)\,dx.
	 \end{align*}
	 For the term $B$ we apply the same H\"{o}lder inequality as above and Lemmas~\ref{lema: control de norma exp L^r respecto de w por norma exp L^r} and~\ref{lema: lema fundamental}. This yields
	 \begin{align*}
	 B&\leq \sum_j \frac{(uv)(Q_j)}{v(Q_j)}\|b_1-b_{1,j}\|_{\mathrm{exp}\,L^{r_1},Q_j,uv}\|b_2-b_{2,j}\|_{\mathrm{exp}\,L^{r_2},Q_j,uv}\|\mathcal{X}_{Q_j}\|_{\Phi,Q_j,uv}\left(\int_{Q_j}fv\right)\\
	 &\leq C\sum_j\frac{(uv)(Q_j)}{v(Q_j)}\|b_1-b_{1,j}\|_{\mathrm{exp}\,L^{r_1},Q_j}\|b_2-b_{2,j}\|_{\mathrm{exp}\,L^{r_2},Q_j}\left(\int_{Q_j}fv\right)\\
	 &\leq C\sum_j\int_{Q_j}f(x)u(x)v(x)\,dx\\
	 &\leq C\int_{\mathbb{R}^n}f(x)u(x)v(x)\,dx.
	 \end{align*}
	
	We turn now to the estimation of $I_2$. Let $x_{Q_j}$ be the centre of the cube $Q_j$, for every $j$. By applying Tchebychev inequality, Tonelli theorem and \eqref{eq: integral cero de hj} we have that
	\begin{align*}
	I_2&\leq \frac{8}{t}\sum_j \int_{\mathbb{R}^n\backslash \Omega^*}|(b_1(x)-b_{1,j})(b_2(x)-b_{2,j})T(h_jv)(x)|u(x)\,dx\\
	&\leq \frac{8}{t}\sum_j \int_{\mathbb{R}^n\backslash Q_j^*}|b_1(x)-b_{1,j}||b_2(x)-b_{2,j}|\left|\int_{Q_j}h_j(y)v(y)(K(x-y)-K(x-x_{Q_j}))\,dy\right|u(x)\,dx\\
	&\leq \frac{8}{t}\sum_j \int_{Q_j} |h_j(y)|v(y)\left(\int_{\mathbb{R}^n\backslash Q_j^*}|b_1(x)-b_{1,j}||b_2(x)-b_{2,j}||K(x-y)-K(x-x_{Q_j})|u(x)\,dx\right)\,dy\\
	&=\frac{8}{t}\sum_j \int_{Q_j} |h_j(y)|v(y)F_j(y)\,dy.
	\end{align*}
	We shall prove that there exists $C>0$ that satisfies
	\begin{equation}\label{eq: teo: desigualdad mixta para Tb caso m=2 - eq1}
	F_j(y)\leq Cu(y), \textrm{ for }  y\in Q_j \textrm{ and every } j.
	\end{equation} 
	Fix $j$ and $y\in Q_j$. Let $\ell_j=\ell(Q_j)/2$, where $\ell(Q_j)$ is the length of the sides of $Q_j$, and $A_{j,k}=\{x: 2^k\ell_j\leq|x-x_{Q_j}|<2^{k+1}\ell_j\}$. By using the smoothness condition of the kernel $K$, we have
	\begin{align*}
	F_j(y)&\leq \sum_{k=1}^\infty \int_{A_{j,k}} |b_1(x)-b_{1,j}||b_2(x)-b_{2,j}|\frac{|y-x_{Q_j}|}{|x-x_{Q_j}|^{n+1}}u(x)\,dx\\
	&\leq C\frac{\ell_j}{(2^{k}\ell_j)^{n+1}}\int_{2^{k+1}Q_j}|b_1(x)-b_{1,j}||b_2(x)-b_{2,j}|u(x)\,dx\\
	&\leq \frac{C2^{-k}}{|2^{k+1}Q_j|}\int_{2^{k+1}Q_j}|b_1(x)-b_{1,j}||b_2(x)-b_{2,j}|u(x)\,dx.
	\end{align*}
	Let $b_{i,j}^{k}=|2^kQ_j|^{-1}\int_{2^kQ_j}b_i$ for $i=1,2$. Thus,
	\begin{align*}
	|b_1(x)-b_{1,j}||b_2(x)-b_{2,j}|&\leq \left|b_1(x)-b_{1,j}^{k+1}\right|\left|b_2(x)-b_{2,j}^{k+1}\right|+\left|b_1(x)-b_{1,j}^{k+1}\right|\left|b_{2,j}^{k+1}-b_{2,j}\right|\\
	&\quad +\left|b_{1,j}^{k+1}-b_{1,j}\right|\left|b_2(x)-b_{2,j}^{k+1}\right|+\left|b_{1,j}^{k+1}-b_{1,j}\right|\left|b_{2,j}^{k+1}-b_{2,j}\right|\\
	&\leq \left|b_1(x)-b_{1,j}^{k+1}\right|\left|b_2(x)-b_{2,j}^{k+1}\right|+C(k+1)\|b_2\|_{\mathrm{Osc}_{\mathrm{exp}\, L^{r_2}}}\left|b_1(x)-b_{1,j}^{k+1}\right|\\
	&\quad +C(k+1)\|b_1\|_{\mathrm{Osc}_{\mathrm{exp}\, L^{r_1}}}\left|b_2(x)-b_{2,j}^{k+1}\right|+C(k+1)^2\|\mathbf b\|,
	\end{align*}
	by virtue of Lemma~\ref{lema: promedios sobre cubos dilatados de b acotada por norma osc}. Then
	\begin{align*}
	F_j(y)&\leq C\sum_{k=1}^\infty \frac{2^{-k}}{|2^{k+1}Q_j|}\int_{2^{k+1}Q_j}\left|b_1(x)-b_{1,j}^{k+1}\right|\left|b_2(x)-b_{2,j}^{k+1}\right|u(x)\,dx\\
	&\quad +C\sum_{k=1}^\infty \frac{2^{-k}(k+1)}{|2^{k+1}Q_j|}\int_{2^{k+1}Q_j}\left|b_1(x)-b_{1,j}^{k+1}\right|u(x)\,dx\\
	&\quad +C\sum_{k=1}^\infty \frac{2^{-k}(k+1)}{|2^{k+1}Q_j|}\int_{2^{k+1}Q_j}\left|b_2(x)-b_{2,j}^{k+1}\right|u(x)\,dx\\
	&\quad +C\sum_{k=1}^\infty \frac{2^{-k}(k+1)^2}{|2^{k+1}Q_j|}\int_{2^{k+1}Q_j}u(x)\,dx.
	\end{align*}
	We now apply a generalized H\"{o}lder inequality to each term above, with respect to $d\zeta= u(x)\,dx$ and with the functions $\varphi_i(t)=e^{t^{r_i}}-1$, for $i=1,2$. This yields
	\[F_j(y)\leq C\left(\inf_{2^{k+1}(Q_j)}u\right)\sum_{k=1}^\infty2^{-k}(k+1)^2\leq Cu(y),\]
	which proves \eqref{eq: teo: desigualdad mixta para Tb caso m=2 - eq1}. By combining this estimate with Lemma~\ref{lema: lema fundamental} we get
	\begin{align*}
	I_2&\leq \frac{C}{t}\sum_j \int_{Q_j}f(y)u(y)v(y)\,dy+\frac{C}{t}\sum_j \frac{uv(Q_j)}{v(Q_j)}\int_{Q_j}f(y)v(y)\,dy\\
	&\leq \frac{C}{t}\int_{\mathbb{R}^n}f(y)u(y)v(y)\,dy.
	\end{align*}\label{pag: estimacion de integral de h_j u v}
	Let us conclude with the estimates of $I_3$ and $I_4$. We need to use here the case $m=1$, which can be proved by following the ideas for the $\mathrm{BMO}$ case (see \cite{Berra-Carena-Pradolini(M)}, Theorem 1). Then
	\begin{align*}
	I_3&\leq uv\left(\left\{x\in \mathbb{R}^n\backslash{\Omega^*}: \frac{|T_{b_1}(\sum_j(b_2-\lambda_{2,j})fv\mathcal{X}_{Q_j})(x)|}{v(x)}>\frac{t}{8}\right\}\right)\\
	&\quad +uv\left(\left\{x\in \mathbb{R}^n\backslash{\Omega^*}: \frac{|T_{b_1}(\sum_j(b_2-\lambda_{2,j})f_{Q_j}^vv\mathcal{X}_{Q_j})(x)|}{v(x)}>\frac{t}{8}\right\}\right)\\
	&=I_{3,1}+I_{3,2}.
	\end{align*}
	By taking $\Phi_1(t)=t(1+\log^+t)^{1/r_1}$ and $\lambda_{2,j}=b_{2,j}$, we have
	\begin{align*}
	I_{3,1}&\leq C\int_{\mathbb{R}^n}\Phi_1\left(\frac{\sum_j|b_2(x)-b_{2,j}|f(x)\mathcal{X}_{Q_j}(x)}{t}\right)u(x)v(x)\,dx\\
	&\leq C\sum_j\int_{Q_j}\Phi_1\left(\frac{|b_2(x)-b_{2,j}|f(x)}{t}\right)u(x)v(x)\,dx.
	\end{align*}
	Since $\Phi^{-1}(t)\varphi_2^{-1}(t)\lesssim \Phi_1^{-1}(t)$ for $t\geq e$, we obtain that 
	\[\Phi_1(st)\leq \Phi(s)+ \varphi_2(t).\]
	Notice that $1=\|b_2\|_{\mathrm{Osc}_{\mathrm{exp}\,L^{r_2}}}\geq \|b_2-b_{2,j}\|_{\mathrm{exp}\,L^{r_2},Q_j}\geq C\|b_2-b_{2,j}\|_{\mathrm{exp}\,L^{r_2},Q_j,uv}$. This implies that 
	\begin{align*}
	I_{3,1}&\leq C\sum_j\int_{Q_j}\left(\Phi\left(\frac{f(x)}{t}\right)+ \varphi_2(|b_2-b_{2,j}|)\right)u(x)v(x)\,dx\\
	&\leq \int_{\mathbb{R}^n}\Phi\left(\frac{f(x)}{t}\right)u(x)v(x)\,dx+C\sum_j uv(Q_j)\\
	&\leq \int_{\mathbb{R}^n}\Phi\left(\frac{f(x)}{t}\right)u(x)v(x)\,dx+\frac{C}{t}\sum_j \frac{uv(Q_j)}{v(Q_j)}\int_{Q_j}f(x)v(x)\,dx\\
	&\leq C\int_{\mathbb{R}^n}\Phi\left(\frac{f(x)}{t}\right)u(x)v(x)\,dx,
	\end{align*}
	where we have used Lemma~\ref{lema: lema fundamental}. On the other hand, since $f_{Q_j}^v/t\leq C$ we have that
	\begin{align*}
	I_{3,2}&\leq C\int_{\mathbb{R}^n}\Phi_1\left(\frac{\sum_j|b_2(x)-b_{2,j}|f_{Q_j}^v\mathcal{X}_{Q_j}(x)}{t}\right)u(x)v(x)\,dx\\
	&\leq C\sum_j\int_{Q_j}\Phi_1\left(C|b_2(x)-b_{2,j}|\right)u(x)v(x)\,dx\\
	&\leq C\sum_j\int_{Q_j}\left(\Phi(C)+\varphi_2(|b_2(x)-b_{2,j}|)\right)u(x)v(x)\,dx,
\end{align*}
and the estimate follows similarly as above. The term $I_4$ can be bounded by interchanging the roles of the components $b_1$ and $b_2$. The proof for $m=2$ is complete. 

\medskip

Suppose now that the result is true for every symbol with $k$ components and consider $\mathbf{b}=~(b_1,b_2,\dots,b_{k+1})$. Let us assume again that $\|b_i\|_{\mathrm{Osc}_{\mathrm{exp}\,L^{r_i}}}=1$ for every $i$. Fixed $t>0$, we perform the Calder\'on-Zygmund decomposition of $f$ with respect to the measure $d\mu=v(x)\,dx$ at level $t$, obtaining a disjoint collection of dyadic cubes $\{Q_j\}_j$. We also decompose $f=g+h$. If $Q_j^*=3Q_j$ and $\Omega^*=\cup_j Q_j^*$ we have that
	\begin{align*}
	uv\left(\left\{x: \frac{|T_\mathbf{b}(fv)(x)|}{v(x)}>t\right\}\right)&= uv\left(\left\{x: \frac{|T_\mathbf{b}(gv)(x)|}{v(x)}>\frac{t}{2}\right\}\right)+uv\left(\left\{x: \frac{|T_\mathbf{b}(hv)(x)|}{v(x)}>\frac{t}{2}\right\}\right)\\
	&\leq uv\left(\left\{x: \frac{|T_\mathbf{b}(gv)(x)|}{v(x)}>\frac{t}{2}\right\}\right)+uv(\Omega^{*})\\
	&\quad +
	uv\left(\left\{x\in \mathbb{R}^n\backslash{\Omega^*}: \frac{|T_\mathbf{b}(hv)(x)|}{v(x)}>\frac{t}{2}\right\}\right)\\
	&=I+II+III.
	\end{align*}
	The estimate of $I$ follows from the strong $(p,p)$ type of $T_\mathbf{b}$ for $A_p$ weights proved in \cite{P-T}, and for $II$ we proceed as in the case $m=2$. 
	
	For $III$ we use Proposition~\ref{propo: Tb en termino de conmutadores con simbolos de menor orden}. If $\lambda_{i,j}$ are constants to be chosen for $1\leq i\leq k+1$ and $j\in\mathbb{N}$,~then
	\begin{align*}
	T_{\mathbf b} (hv)(x)&=\sum_j T_{\mathbf b} (h_jv)(x)\\ 
	&=\sum_j\left(\prod_{i=1}^{k+1}(b_i(x)-\lambda_{i,j})\right)T(h_jv)(x)-\sum_j T\left(\left(\prod_{i=1}^{k+1}(b_i-\lambda_{i,j})\right)h_jv\right)(x)\\
	&\quad 
	-\sum_{\sigma\in S_{k+1},\,0<|\sigma|<k+1 }T_{\mathbf b_\sigma}\left(\sum_j \left(\prod_{i=1}^{k+1}(b_i-\lambda_{i,j})^{\bar\sigma_i}\right)h_jv\right)(x),
	\end{align*}\label{pag: estimacion de III caso general}
    which implies that
	\begin{align*}
	III&\leq 	uv\left(\left\{x\in \mathbb{R}^n\backslash{\Omega^*}: \sum_j \left(\prod_{i=1}^{k+1}|b_i(x)-\lambda_{i,j}|\right)\frac{|T(h_jv)(x)|}{v(x)}>\frac{t}{2^{k+2}}\right\}\right)\\
	&\quad +  uv\left(\left\{x\in \mathbb{R}^n\backslash{\Omega^*}: \sum_j \frac{\left|T\left(\left(\prod_{i=1}^{k+1}(b_i-\lambda_{i,j})\right)h_jv\right)(x)\right|}{v(x)}>\frac{t}{2^{k+2}}\right\}\right)\\
	&\quad +\sum_{\sigma\in S_{k+1}, \,0<|\sigma|<k+1}uv\left(\left\{x\in \mathbb{R}^n\backslash{\Omega^*}: \frac{\left|T_{\mathbf b_\sigma}\left(\sum_j\left(\prod_{i=1}^{k+1}(b_i-\lambda_{i,j})^{\bar\sigma_i}\right)h_jv\right)(x)\right|}{v(x)}>\frac{t}{2^{k+2}}\right\}\right)\\
	&= I_1+I_2+I_3.
	\end{align*}
	Let us first estimate $I_1$. Take $\lambda_{i,j}=b_{i,j}$ for every $j$ and every $1\leq i\leq k+1$. By applying Tchebychev inequality we have that
	\begin{align*}
	I_1&\leq \sum_j \int_{\mathbb{R}^n\backslash \Omega^*}\left(\prod_{i=1}^{k+1}|b_i(x)-b_{i,j}|\right)\left|\int_{Q_j}h_j(y)v(y)K(x-y)\,dy\right|u(x)\,dx\\
	&\leq \sum_j \int_{Q_j}|h_j(y)|v(y)\left(\int_{\mathbb{R}^n\backslash Q_j^*}|K(x-y)-K(x-x_{Q_j})|\left(\prod_{i=1}^{k+1}|b_i(x)-b_{i,j}|\right)u(x)\,dx\right)\,dy\\
	&=\sum_j \int_{Q_j}|h_j(y)|v(y)F_{j,k}(y)\,dy.
	\end{align*}\label{pag: estimacion de I_1 caso general}
	We shall prove that $F_{j,k}(y)\leq Cu(y)$, for $y\in Q_j$ and $C$ independent of $j$. Indeed, by using the smoothness condition of $K$ \eqref{eq:prop del nucleo} we get
	\begin{align*}
	F_{j,k}(y)&\leq \sum_{\ell=0}^\infty \int_{A_{j,\ell}} \frac{|y-x_{Q_j}|}{|x-x_{Q_j}|^{n+1}}\left(\prod_{i=1}^{k+1}|b_i(x)-b_{i,j}|\right)u(x)\,dx\\
	&\leq C\sum_{\ell=0}^\infty \frac{2^{-\ell}}{|2^{\ell+1}Q_j|}\int_{2^{\ell+1}Q_j} \left(\prod_{i=1}^{k+1}|b_i(x)-b_{i,j}|\right)u(x)\,dx.
	\end{align*}
	Let $b_{i,j}^{\ell}=|2^{\ell}Q_j|^{-1}\int_{2^{\ell}Q_j}b_i$. By virtue of Lemmas~\ref{lema: promedios sobre cubos dilatados de b acotada por norma osc} and~\ref{lema: producto de sumas} we obtain
	\begin{align*}
	\prod_{i=1}^{k+1}|b_i(x)-b_{i,j}|&= \sum_{\sigma\in S_{k+1}}\prod_{i=1}^{k+1}|b_i-b_{i,j}^{\ell+1}|^{\sigma_i}|b_{i,j}^{\ell+1}-b_{i,j}|^{\bar \sigma_i}\\
	&\leq C(\ell+1)^{k+1}\prod_{i=1}^{k+1}|b_i-b_{i,j}^{\ell+1}|^{\sigma_i}.
	\end{align*}
	We apply now generalized H\"{o}lder inequality with functions $\varphi_i(t)=e^{t^{r_i}}-1$, $1\leq i\leq k+1$, with respect to the measure $d\zeta=u(x)\,dx$. This yields
	\begin{align*}
	F_{j,k}(y)&\leq C\sum_{\ell=0}^\infty\frac{u(2^{\ell+1}Q_j)}{|2^{\ell+1}Q_j|}2^{-\ell}(\ell+1)^{k+1}\times\\
	&\quad \times \sum_{\sigma\in S_{k+1}}\left(\prod_{i: \sigma_i=1}\|b_i-b_{i,j}^{\ell+1}\|_{\varphi_i,2^{\ell+1}Q_j,u}\right)\left(\prod_{i: \sigma_i=0}\|\mathcal{X}_{2^{\ell+1}Q_j}\|_{\varphi_i,2^{\ell+1}Q_j,u}\right)\\
	&\leq C[u]_{A_1} \sum_{\ell=0}^\infty2^{-\ell}(\ell+1)^{k+1}u(y)\\
	&=Cu(y).
	\end{align*}
	From here the estimate follows in the same way as in page~\pageref{pag: estimacion de integral de h_j u v}.
	
	Let us focus now on $I_2$. Considering again $\lambda_{i,j}=b_{i,j}$, Theorem~\ref{teo: desigualdad mixta para T} allows us to get
	\begin{align*}
	I_2&\leq \frac{C}{t}\sum_j \int_{Q_j} \left(\prod_{i=1}^{k+1} |b_i(x)-b_{i,j}|\right)|h_j(x)|u(x)v(x)\,dx\\
	&\leq \frac{C}{t}\sum_j \int_{Q_j} \left(\prod_{i=1}^{k+1} |b_i(x)-b_{i,j}|\right)f(x)u(x)v(x)\,dx\\
	&\quad +\frac{C}{t}\sum_j \int_{Q_j} \left(\prod_{i=1}^{k+1} |b_i(x)-b_{i,j}|\right)f_{Q_j}^vu(x)v(x)\,dx\\
	&=I_2^1+I_2^2.
	\end{align*}\label{pag: estimacion de I2}
	For $I_2^1$, observe that
	\[\left(\prod_{i=1}^{k+1}\varphi_i^{-1}(t)\right)\Phi^{-1}(t)\lesssim t,\]
	for $t\geq e$. By applying generalized H\"{o}lder inequality with respect to $uv$, Lemma~\ref{lema: control de norma exp L^r respecto de w por norma exp L^r} and \eqref{eq: equivalencia de norma luxemburgo con integral} we obtain that
	\begin{align*}
	I_2^1&\leq \frac{C}{t}\sum_j (uv)(Q_j)\left(\prod_{i=1}^{k+1}\|b_i-b_{i,j}\|_{\varphi_i,Q_j,uv}\right)\|f\|_{\Phi,Q_j,uv}\\
	&\leq \frac{C}{t}\sum_j(uv)(Q_j)\left(\prod_{i=1}^{k+1}\|b_i-b_{i,j}\|_{\varphi_i,Q_j}\right)\left\{t+\frac{t}{(uv)(Q_j)}\int_{Q_j}\Phi\left(\frac{f(x)}{t}\right)u(x)v(x)\,dx\right\}\\
	&\leq C\sum_j \left(\frac{uv(Q_j)}{v(Q_j)}\frac{1}{t}\int_{Q_j}f(x)v(x)\,dx+\int_{Q_j}\Phi\left(\frac{f(x)}{t}\right)u(x)v(x)\,dx\right)\\
	&\leq C\int_{\mathbb{R}^n}\Phi\left(\frac{f(x)}{t}\right)u(x)v(x)\,dx,
	\end{align*}
	by virtue of Lemma~\ref{lema: lema fundamental}.
	For $I_2^2$ we apply again  H\"{o}lder inequality to get
	\begin{align*}
	I_2^2&\leq \frac{C}{t}\sum_j\frac{(uv)(Q_j)}{v(Q_j)}\left(\int_{Q_j}fv\right)\left(\prod_{i=1}^{k+1}\|b_i-b_{i,j}\|_{\varphi_i,Q_j,uv}\right)\|\mathcal{X}_{Q_j}\|_{\Phi,Q_j,uv}\\
	&\leq \sum_j\frac{C}{t}\left(\int_{Q_j}fuv\right)\left(\prod_{i=1}^{k+1}\|b_i-b_{i,j}\|_{\varphi_i,Q_j}\right)\\
	&\leq \frac{C}{t}\left(\int_{\mathbb{R}^n}fuv\right).
	\end{align*}
	Finally, we estimate $I_3$. Fix $\sigma\in S_{k+1}$ such that $0<|\sigma|<k+1$. Let $\Phi_{\sigma}(t)=t(1+\log^+t)^{1/r_{\sigma}}$, where
	\[\frac{1}{r_\sigma}=\sum_{i: \sigma_i=1}\frac{1}{r_i}.\]
	By using the inductive hypothesis we conclude
	\[uv\left(\left\{ \frac{\left|T_{\mathbf b_\sigma}\left(\sum_j\left(\prod_{i=1}^{k+1}(b_i-\lambda_{i,j})^{\bar\sigma_i}\right)h_jv\right)(x)\right|}{v(x)}>\frac{t}{2^{k+1}}\right\}\right)\leq I_3^1(\sigma)+I_3^2(\sigma),\]
	where
	\[I_3^1(\sigma)=C\sum_j\int_{Q_j}\Phi_\sigma\left(\frac{\left(\prod_{i=1}^{k+1}|b_i-\lambda_{i,j}|^{\bar\sigma_i}\right)f}{t}\right)uv\]
	and
	\[I_3^2(\sigma)=C\sum_j\int_{Q_j}\Phi_\sigma\left(\frac{\left(\prod_{i=1}^{k+1}|b_i-\lambda_{i,j}|^{\bar\sigma_i}\right)f_{Q_j}^v}{t}\right)uv.\]\label{pag: estimacion de I3 sigma}
	To estimate $I_3^1(\sigma)$ observe that
	\[\Phi^{-1}(t)\left(\prod_{i:\sigma_i=0}\varphi_i^{-1}(t)\right)\approx \frac{t}{(1+\log^+t)^{1/r}}(1+\log^+t)^{1/r-1/r_\sigma}=\frac{t}{(1+\log^+t)^{1/r_\sigma}}=\Phi_\sigma^{-1}(t),\]
	which implies that
	\[\Phi_\sigma\left(\frac{\left(\prod_{i=1}^{k+1}|b_i-\lambda_{i,j}|^{\bar\sigma_i}\right)f}{t}\right)\leq \Phi\left(\frac{f}{t}\right)+\sum_{i: \sigma_i=0}\varphi_i(|b_i-b_{i,j}|).\]
	Since $1=\|b_i\|_{\mathrm{Osc}_{\mathrm{exp}\,L^{r_i}}}\geq \|b_i-b_{i,j}\|_{\mathrm{exp}\,L^{r_i},Q_j}\geq C\|b_i-b_{i,j}\|_{\mathrm{exp}\,L^{r_i},Q_j,uv}$ for every $i$ such that $\sigma_i=0$, we have
	\begin{align*}
	I_3^1(\sigma)&\leq C\sum_j\int_{Q_j}\Phi\left(\frac{f}{t}\right)uv+\sum_j\sum_{i: \sigma_i=0}\int_{Q_j}\varphi_i(|b_i-b_{i,j}|)uv\\
	&\leq C\int_{\mathbb{R}^n}\Phi\left(\frac{f}{t}\right)uv+C\sum_j \frac{(uv)(Q_j)}{v(Q_j)}\int_{Q_j}\frac{f}{t}v\\
	&\leq C\int_{\mathbb{R}^n}\Phi\left(\frac{f}{t}\right)uv,
	\end{align*}
	by Lemma~\ref{lema: lema fundamental}.
	On the other hand, since $f_{Q_j}^v/t\leq C$
	\begin{align*}
	I_3^2(\sigma)&\leq C\sum_j uv(Q_j)+C\sum_j\sum_{i: \sigma_i=0} \int_{Q_j}\varphi_i(|b_i-b_{i,j}|)uv\\
	&\leq C\sum_j\frac{(uv)(Q_j)}{v(Q_j)}\frac{1}{t}\int_{Q_j}fv+C(uv)(Q_j)\\
	&\leq C\int_{\mathbb{R}^n}\Phi\left(\frac{f}{t}\right)uv.
	\end{align*}	
	Therefore the result holds for every symbol $\mathbf b$ with $k+1$ components. The proof of Theorem~\ref{teo: desigualdad mixta para Tb caso CZO} is complete. \qedhere
\end{proof}

\section{Operators with kernel of Hörmander type}\label{seccion: caso Hormander}

This section will be devoted to prove the strong $(p,p)$ type of $T_\mathbf{b}$ on the Hörmander setting as well as a mixed inequality for these operators. When $b$ is a single symbol belonging to $\mathrm{BMO}$ and $T=K*f$, the first result follows by combining a Coifman type inequality proved in \cite{LRdlT} with the weighted strong $(p,p)$ type of $M_\varphi$ obtained in \cite{B-D-P}, where $\varphi$ is a Young function with certain properties. One of the keys for the proof of the aforementioned Coifman inequality is a pointwise relation between the lower order commutators of $T_b^m$ and the sharp-$\delta$ maximal function $M^{\#}_\delta f(x)=((M^{\#} f^\delta )(x))^{1/\delta}$, where
\[M^{\#} f(x)=\sup_{Q\ni x} \frac{1}{|Q|}\int_Q |f(y)-f_Q|\,dy\] 
and $f_Q=|Q|^{-1}\int_Q f$. It is well-known (see, for example, \cite{javi}) that 
\begin{equation}\label{eq: equivalencia M sharp con infimo}
M^{\#}f(x)\approx \sup_Q \inf_{a\in\mathbb{C}}\frac{1}{|Q|}\int_Q |f(y)-a|\,dy
\end{equation}
and that for every $0<p<\infty$ and $w\in A_\infty$ the inequality

\begin{equation}\label{eq: desigualdad de Fefferman-Stein}
\int_{\mathbb{R}^n} (Mf(x))^pw(x)\,dx\leq C\int_{\mathbb{R}^n} (M^{\#}f(x))^pw(x)\,dx,
\end{equation}
holds whenever the left-hand side is finite.

The following proposition gives the pointwise relation involving $M^{\#}_\delta$ on the multilinear context. 

   \begin{propo}\label{propo: M sharp delta en terminos de conmutadores de orden menor}
    	Let $m\in\mathbb{N}$, $r_i\geq 1$ for $1\leq i\leq m$ and $1/r=\sum_{i=1}^m 1/r_i$. Let $\eta, \phi$ be Young functions that verify $\bar{\eta}^{-1}(t)\phi^{-1}(t)(\mathrm{log}\, t)^{1/r}\lesssim t$, for $t\geq e$. Let $T$ be an operator with kernel $K\in H_{\phi,m}$ and $\mathbf{b}=(b_1,b_2,\dots,b_m)$, where $b_i\in \mathrm{Osc}_{\mathrm{exp}\,L^{r_i}}$ for $1\leq i\leq m$. Then the inequality
    	\begin{equation}\label{eq: estimacion M^sharp_delta para Tbm con nucleo Hormander}
    	M^{\#}_\delta(T_\mathbf{b}f)(x)\leq C\|\mathbf{b}\|M_{\bar\eta}f(x)+C\!\!\sum_{\sigma\in S_m,\,|\sigma|<m }\|\mathbf{b}_{\bar \sigma}\|M_{\varepsilon}(T_{\mathbf{b}_\sigma}f)(x)
    	\end{equation}
    	 holds for every $0<\delta<\varepsilon<1$.
    \end{propo}

\begin{proof}
	Fix $0<\delta<\varepsilon<1$. By virtue of Proposition~\ref{propo: Tb en terminos de conmutadores con simbolos de menor orden aplicado a f} we can write
	\[T_\mathbf{b}f(x)=(-1)^mT\left(\left(\prod_{i=1}^m(b_i-\lambda_i)\right)f\right)(x)+\sum_{\sigma\in S_m,\, |\sigma|<m}(-1)^{m-1-|\sigma|}\prod_{i=1}^m(b_i-\lambda_i)^{\bar\sigma_i}T_{\mathbf{b}_\sigma}f(x),\]
	where $\lambda_i$ are constants to be chosen later for $1\leq i\leq m$.
	
	By homogeneity we can assume, without loss of generality, that $\|b_i\|_{\mathrm{Osc}_{\mathrm{exp}\,L^{r_i}}}=1$ for every $1\leq i\leq m$. Fix $x\in\mathbb{R}^n$ and $B=B(x_B, R)$ a ball that contains $x$. By using \eqref{eq: equivalencia M sharp con infimo}, we decompose $f=f\mathcal{X}_{2B}+f\mathcal{X}_{(2B)^c}=f_1+f_2$ and take a constant $a_B$ that we shall choose properly. Thus
	\begin{align*}
	\left(\frac{1}{|B|}\int_B \left||T_{\mathbf{b}}f(y)|^\delta-|a_B|^\delta\right|\,dy\right)^{1/\delta}&\leq \left(\frac{1}{|B|}\int_B |T_{\mathbf{b}}f(y)-a_B|^\delta\,dy\right)^{1/\delta}\\
	&\leq C_\delta\sum_{\sigma\in S_m,\,|\sigma|<m}\left(\frac{1}{|B|}\int_B \left(\prod_{i=1}^m|b_i-\lambda_i|^{\delta\bar \sigma_i}\right)|T_{\mathbf{b}_\sigma}f|^\delta\right)^{1/\delta}+\\
	&\quad  +C_\delta\left(\frac{1}{|B|}\int_B \left|T\left(\left(\prod_{i=1}^m(b_i-\lambda_i)\right)f_1\right)(y)\right|^\delta dy\right)^{1/\delta}+\\
	&\quad  +C_\delta\left(\frac{1}{|B|}\int_B \left|T\left(\left(\prod_{i=1}^m(b_i-\lambda_i)\right)f_2\right)(y)-a_B\right|^\delta dy\right)^{1/\delta}\\
	&=I+II+III.
	\end{align*}
	Let us first estimate $I$. If $q=\varepsilon/\delta>1$, then for every $\sigma\in S_m$ with at least one component equal to 0 we have that for $t\geq t_0$
	\[t^{1/q}\prod_{i: \sigma_i=0}(\mathrm{log}\,t)^{1/r_i}\leq t^{1/q}(\mathrm{log}\,t)^{1/r}\lesssim t.\]
	By taking $\lambda_i = (b_i)_{2B}$ and applying generalized Hölder inequality, we get
	\begin{align*}
	I&\leq \sum_{\sigma\in S_m,\, |\sigma|<m}\left(\prod_{i:\sigma_i=0}\left\|(b_i-(b_i)_{2B})^\delta\right\|_{\mathrm{exp}\,L^{r_i},B}^{1/\delta}\right)\left(\frac{1}{|B|}\int_B |T_{\mathbf{b}_\sigma}f(y)|^{\delta q}dy\right)^{1/(\delta q)}\\
	&\leq \sum_{\sigma\in S_m,\, |\sigma|<m}\left(\prod_{i:\sigma_i=0}\left\|b_i-(b_i)_{2B}\right\|_{\mathrm{exp}\,L^{r_i},B}\right)M_{\varepsilon}\left(T_{\mathbf{b}_\sigma}f\right)(x)\\
	&\leq C \sum_{\sigma\in S_m,\, |\sigma|<m} M_\varepsilon(T_{\mathbf{b}_\sigma}f)(x).
	\end{align*}
	
	We shall now estimate $II$. Since $H_{\phi}\subset H_1$, $T$ is of weak $(1,1)$ type. By applying Kolmogorov inequality we have that
	\begin{align*}
	II&\leq \frac{C}{|B|}\int_B\left(\prod_{i=1}^m|b_i(y)-(b_i)_{2B}|\right)|f(y)|\,dy\\
	&\leq \frac{C}{|2B|}\int_{2B}\left(\prod_{i=1}^m|b_i(y)-(b_i)_{2B}|\right)|f(y)|\,dy\\
	&\leq C \left(\prod_{i=1}^m \|b_i-(b_i)_{2B}\|_{\mathrm{exp}\,L^{r_i}, 2B}\right)\|\mathcal{X}_{2B}\|_{\phi, 2B}M_{\bar\eta}f(x)\\
	&\leq CM_{\bar\eta}f(x),
	\end{align*}
	where we have used generalized Hölder inequality with functions $\varphi_i(t)=e^{t^{1/r_i}}-1$, $1\leq i\leq m$, $\phi(t)$ and $\bar\eta$, since
	\[\bar{\eta}^{-1}(t)\phi^{-1}(t)(\mathrm{log}\, t)^{1/r}\lesssim t\]
	for $t\geq e$.
	
	Finally, we estimate $III$. Pick $a_B=T\left(\left(\prod_{i=1}^m (b_i-(b_i)_{2B})\right)f_2\right)(x_B)$. Jensen inequality yields
	\begin{align*}
	III&\leq \frac{1}{|B|}\int_B \left|T\left(\left(\prod_{i=1}^m\left(b_i-(b_i)_{2B}\right)\right)f_2\right)(y)-T\left(\left(\prod_{i=1}^m\left(b_i-(b_i)_{2B}\right)\right)f_2\right)(x_B)\right|\,dy\\
	&=\frac{1}{|B|}\int_B F(y)\,dy.
	\end{align*}
	We define, for $j\geq 1$, the sets
	\[A_j=\{y: 2^jR\leq |x_B-y|<2^{j+1}R\}.\]
	By using the integral representation of $T$, we have
	\[F(y)=\int_{\mathbb{R}^n}\prod_{i=1}^m |b_i(z)-(b_i)_{2B}||f_2(z)||K(y-z)-K(x_B-z)|\,dz.\]
	Lemmas~\ref{lema: promedios sobre cubos dilatados de b acotada por norma osc} and~\ref{lema: producto de sumas} imply
	\begin{align*}
	\prod_{i=1}^m |b_i(z)-(b_i)_{2B}|&=\prod_{i=1}^m \left(|b_i(z)-(b_i)_{2^{j+1}B}|+|(b_i)_{2^{j+1}B}-(b_i)_{2B}|\right)\\
	&\leq \prod_{i=1}^m \left(|b_i(z)-(b_i)_{2^{j+1}B}|+Cj\right)\\
	&=\sum_{\sigma\in S_m}\prod_{i=1}^m \left|b_i(z)-(b_i)_{2^{j+1}B}\right|^{\sigma_i}(Cj)^{\bar\sigma_i}\\
	&\leq C \sum_{\sigma\in S_m}j^{m-|\sigma|}\prod_{i=1}^m |b_i(z)-(b_i)_{2^{j+1}B}|^{\sigma_i}.
	\end{align*}
	Thus, by applying again Hölder inequality
	\begin{align*}
	F(y)&\lesssim \sum_{\sigma\in S_m}\sum_{j=1}^\infty j^{m-|\sigma|}\frac{(2^j R)^n}{|2^{j+1}B|}\int_{2^{j+1}B}\left(\prod_{i=1}^m|b_i(z)-(b_i)_{2^{j+1}B}|^{\sigma_i}\right)\times\\
	&\qquad\times|K(y-z)-K(x_B-z)|\mathcal{X}_{A_j}(z)|f_2(z)|\,dz\\
	&\leq C\sum_{\sigma\in S_m} \sum_{j=1}^\infty j^{m-|\sigma|}\left[(2^jR)^n \left(\prod_{i:\sigma_i=1} \|b_i-(b_i)_{2^{j+1}B}\|_{\mathrm{exp}\,L^{r_i}, 2^{j+1}B}\right)\|f\|_{\bar \eta, 2^{j+1}B}\right. \times\\
	&\quad\left. \phantom{\prod_{i=1}^m}\times\|K(\cdot - (x_B-y)-K(\cdot))\|_{\phi, |z|\sim 2^jR}\right]\\
	&\leq CM_{\bar \eta}f(x),
	\end{align*}
	since $K\in H_{\phi,m}$ implies $K\in H_{\phi,m-|\sigma|}$.
	By combining these estimates we get 
	\[M^{\#}_\delta(T_\mathbf{b}f)(x)\leq CM_{\bar\eta}f(x)+C\sum_{\sigma\in S_m,\, |\sigma|<m}M_{\varepsilon}(T_{\mathbf{b}_\sigma}f)(x).\qedhere\]
\end{proof}

The following result contains a Coifman type estimate involving $T_\mathbf{b}$. The proof can be performed by using the proposition above and following similar arguments as in Theorem 3.3 in \cite{LRdlT}. We shall omit the details.

\begin{teo}\label{teo: desigualdad tipo Coifman para Tb^m Hormander}
	Let $m\in\mathbb{N}$, $r_i\geq 1$ for every $1\leq i\leq m$ and $1/r=\sum_{i=1}^m 1/r_i$. Let $\eta$ and $\phi$ be Young functions that verify $\bar{\eta}^{-1}(t)\phi^{-1}(t)(\mathrm{log}\, t)^{1/r}\lesssim t$, for $t\geq e$. Let $T$ be an operator with kernel $K\in H_{\phi,m}$ and $\mathbf{b}=(b_1,b_2,\dots,b_m)$ where $b_i\in \mathrm{Osc}_{\mathrm{exp}\,L^{r_i}}$, for $1\leq i\leq m$. Then for every $0<p<\infty$ and $w\in A_\infty$ there exists a positive constant $C$ such that the inequality
	\[\int_{\mathbb{R}^n}|T_\mathbf{b}f(x)|^pw(x)\,dx\leq C\|\mathbf{b}\|^p\int_{\mathbb{R}^n}(M_{\bar\eta}f(x))^pw(x)\,dx\]
	holds for every bounded function $f$ with compact support, provided the left-hand side is finite. 
\end{teo}	
When we consider a single symbol in BMO, the result above was proved in \cite{LRdlT}. The corresponding result for commutators of CZO  was obtained in \cite{Perez_Sharp97}.

Now we are in a position to give the proof of the strong $(p,p)$ type for $T_\mathbf{b}$.

\begin{proof}[Proof of Theorem~\ref{teo: tipo fuerte de Tb Hormander}]
	We start again by assuming that $\|b_i\|_{\mathrm{Osc}_{\mathrm{exp}\,L^{r_i}}}=1$ for each $1\leq i\leq m$. It is well known under the hypotheses of this theorem (see \cite{B-D-P}, Theorem 2.5) we have that
	\[\|(M_{\bar\eta}f)w\|_{L^p}\leq C\|fw\|_{L^p}.\]
	The conclusion follows immediately from Theorem~\ref{teo: desigualdad tipo Coifman para Tb^m Hormander} applied with $w^p$.
\end{proof}

We close this section with the proof of mixed inequalities for $T_\mathbf{b}$. We shall need a mixed inequality involving the operator $T$, which was set and proved in \cite{Berra-Carena-Pradolini(M)}.	

\begin{teo}\label{teo: desigualdad mixta para T Hormander}
	Let $1<q<\infty$ and $q^2/(2q-1)<\beta<q$. Assume that $w\in A_1\cap \mathrm{RH}_s$ for some $s>1$ and $v^{\alpha}\in A_{(q/\beta)'}(u)$, where $\alpha=\beta(q-1)/(q-\beta)$. Let $T$ be an operator with kernel $K\in H_\phi$, where $\phi$ is a Young function that verifies $\bar{\phi}\in B_\rho$ for every $\rho\geq \min\{\beta,s\}$. Then there exists $C>0$ such that the inequality
	\[uv\left(\left\{x\in\mathbb{R}^n: \left|\frac{T(fv)(x)}{v(x)}\right|>t\right\}\right)\leq \frac{C}{t}\int_{\mathbb{R}^n} |f(x)|u(x)v(x)\,dx\]
\end{teo}
holds for every $t>0$.

\begin{proof}[Proof of Theorem~\ref{teo: desigualdad mixta para Tb multilineal Hormander}]
	We proceed by induction on $m$. When $m=1$ we have a symbol with a single component; the proof can be achieved by following the corresponding version of the $\mathrm{BMO}$ case (\cite{Berra-Carena-Pradolini(M)}, Theorem 4). Assume now that the result holds for any symbol with $k$ components and let us prove it for $\mathbf{b}=(b_1,b_2,\dots, b_{k+1})$. Also assume without loss of generality that $f$ is bounded with compact support and that $\|b_i\|_{\mathrm{Osc}_{\mathrm{exp}\, L^{r_i}}}=1$ for each $1\leq i\leq k+1$. Since $\alpha>1$ we have $v\in A_{(q/\beta)'}(u)\subset A_\infty(u)$, and this implies that $d\mu(x)=v(x)\,dx$ is doubling. We perform the Calderón-Zygmund decomposition of $f$ at height $t>0$, with respect to $v$, obtaining a disjoint collection of dyadic cubes $\{Q_j\}_j$ that verify
	\[t<f_{Q_j}^v\leq Ct,\]
	for each $j$.  We also split $f=g+h$. Let $Q_j^*=2c\sqrt{n}Q_j$, where $c=\min\{c_{\bar\eta},c_\phi\}$ and these constants are the parameters appearing in the conditions $H_{\phi,m}$ and $H_{\bar\eta}$ for $K$.  If $\Omega^*=\bigcup_j Q_j^*$, then 
	
	\begin{align*}
	uv\left(\left\{x: \frac{|T_\mathbf{b}(fv)(x)|}{v(x)}>t\right\}\right)& \leq uv\left(\left\{x: \frac{|T_\mathbf{b}(gv)(x)|}{v(x)}>\frac{t}{2}\right\}\right)+uv(\Omega^{*})\\
	&\quad +
	uv\left(\left\{x\in \mathbb{R}^n\backslash{\Omega^*}: \frac{|T_\mathbf{b}(hv)(x)|}{v(x)}>\frac{t}{2}\right\}\right)\\
	&=I+II+III.
	\end{align*}
	Notice that the hypothesis on $v$ yields $v^{\alpha(1-(q/\beta))}\in A_{q/\beta}(u)$, that is $v^{1-q}\in A_{q/\beta}(u)$ and therefore $w^q=uv^{1-q}\in A_{q/\beta}$. By applying Tchebychev inequality with $q$ and Theorem~\ref{teo: tipo fuerte de Tb Hormander} we get
	\begin{align*}
	I&\leq \frac{C}{t^q}\int_{\mathbb{R}^n}|T_\mathbf{b}(gv)(x)|^qw(x)^q\,dx\\
	&\leq \frac{C}{t^q}\int_{\mathbb{R}^n}|g(x)v(x)|^qw(x)^q\,dx\\
	&=\frac{C}{t^q}\int_{\mathbb{R}^n}|g|^qu(x)v(x)\,dx,
	\end{align*}
	provided $\|(T_\mathbf{b}(gv))w\|_{L^q}$ is finite. From this point the estimate follows exactly as in page~\pageref{pag: estimacion de I}.
	
	Let us prove that $\|(T_\mathbf{b}(gv))w\|_{L^q}$ is indeed finite. As a first step assume that $w^q$ and every component function of $\mathbf{b}$ are in $L^\infty$. Inequality~\ref{eq: representacion de Tb con b multilineal} implies that
	\[|T_\mathbf{b}(gv)(x)|\leq \sum_{\sigma\in S_{k+1}} \left(\prod_{i=1}^{k+1}\left|b_i^{\sigma_i}(x)\right|\right)\left|T\left(\left(\prod_{i=1}^{k+1}b_i^{\bar{\sigma}_i}\right)gv\right)(x)\right|,\]
	and using the strong $(q,q)$ type of $T$ we obtain that
	\begin{align*}
	\int_{\mathbb{R}^n}|T_{\mathbf{b}}(gv)(x)|^qw^q(x)\,dx&\leq C\|w^q\|_{L^\infty}\left(\prod_{i=1}^{k+1}\|b_i\|_{L^\infty}^q\right)\int_{\mathbb{R}^n}|g(x)|^qv^q(x)\,dx\\
	&\leq C\|w^q\|_{L^\infty}\left(\prod_{i=1}^{k+1}\|b_i\|_{L^\infty}^q\right)\|g\|_{L^\infty}^q\int_{Q_0}v^q(x)\,dx\\
	&<\infty,
	\end{align*}
	since condition $\beta>q^2/(2q-1)$ implies that $v^q$ is a weight and, consequently, locally integrable. The additional assumption that $w^q$ and every $b_i$ are in $L^\infty$ can be removed by following standard arguments similarly as we did in the proof of Theorem~\ref{teo: desigualdad tipo Coifman para Tb^m Hormander}. 
	
	The estimate of $II$ follows exactly as in page~\pageref{pag: estimacion de II}. For $III$ we proceed as in the proof of Theorem~\ref{teo: desigualdad mixta para Tb caso CZO} in page~\pageref{pag: estimacion de III caso general} to get
	\[III\leq I_1+I_2+I_3.\]

	For $I_1$, as we did in page~\pageref{pag: estimacion de I_1 caso general}, we get
	\[I_1\leq \sum_j \int_{Q_j}|h_j(y)|v(y)F_{j,k}(y)\,dy.\]
	We shall prove that $F_{j,k}(y)\leq Cu(y)$, for every $y\in Q_j$ and with $C$ independent of $j$. 
	By setting $b_{i,j}^{\ell}=|2^{\ell}Q_j|^{-1}\int_{2^{\ell}Q_j}b_i$. By Lemmas~\ref{lema: promedios sobre cubos dilatados de b acotada por norma osc} and~\ref{lema: producto de sumas} we can conclude that
	\begin{align*}
	\prod_{i=1}^{k+1}|b_i(x)-b_{i,j}|&= \sum_{\sigma\in S_{k+1}}\prod_{i=1}^{k+1}\left|b_i-b_{i,j}^{\ell+1}\right|^{\sigma_i}\left|b_{i,j}^{\ell+1}-b_{i,j}\right|^{\bar \sigma_i}\\
	&\leq C(\ell+1)^{k+1}+\sum_{\sigma\in S_{k+1},\,|\sigma|>0}(\ell+1)^{k+1-|\sigma|}\prod_{i: \sigma_i=1}\left|b_i-b_{i,j}^{\ell+1}\right|.
	\end{align*}
	Let $A_{j,\ell}=\{x: 2^{\ell-1}r_j<|x-x_{Q_j}|\leq 2^\ell r_j\}$, where $r_j=c\sqrt{n}\ell(Q_j)$. Then we have that 	\begin{align*}
	F_{j.k}(y)&\leq C\sum_{\ell=0}^\infty(\ell+1)^{k+1}\int_{A_{j,\ell}}|K(x-y)-K(x-x_{Q_j})|u(x)\,dx\\
	&\quad +\sum_{\sigma\in S_{k+1},\,|\sigma|>0}\sum_{\ell=0}^\infty(\ell+1)^{k+1-|\sigma|}\int_{A_{j,\ell}}\prod_{i:\sigma_i=1} \left|b_i(x)-b_{i,j}^{\ell+1}\right||K(x-y)-K(x-x_{Q_j})|u(x)\,dx\\
	&=F_{j,k}^1(y)+\sum_{\sigma\in S_{k+1},\,|\sigma|>0} F_{j,k}^\sigma(y).
	\end{align*}
	
	Let $B_j^\ell =B(x_{Q_j}, 2^{\ell}r_j)$. Recall that from Remark~\ref{obs: H_{phi,m} implica H_{eta,m}} we have that $K\in H_{\eta,k+1}$. Then we apply generalized Hölder inequality with functions $\eta$ and $\bar\eta$ to get
	\begin{align*}
	F_{j,k}^1(y)&\leq C\sum_{\ell=0}^\infty\ell^{k+1}(2^{\ell}r_j)^n\|(K(\cdot-y)-K(\cdot-x_{Q_j}))\mathcal{X}_{A_{j,\ell}}\|_{\eta, B_j^\ell}\|u\|_{\bar\eta, B_j^\ell}\\
	&=C\sum_{\ell=0}^\infty\ell^{k+1}(2^{\ell}r_j)^n\|K(\cdot-y)-K(\cdot-x_{Q_j})\|_{\eta, |x|\sim 2^{\ell}r_j}\|u\|_{\bar\eta, B_j^\ell}\\
	&\leq C_\eta M_{\bar\eta}u(y)\\
	&\leq Cu(y),
	\end{align*}
	by virtue of Lemma~\ref{lema: control puntual de MPhi u por u}.
	
	On the other hand, $K\in H_{\phi,k+1}$ implies $K\in H_{\phi,l}$ for every $0\leq l\leq k+1$. Fixed $\sigma\in S_{k+1}$, we apply generalized Hölder inequality with functions $\varphi_i(t)=e^{t^{r_i}}-1$, $1\leq i\leq k+1$, $\phi$ and $\bar\eta$, since
	\[\bar{\eta}^{-1}(t)\phi^{-1}(t)\prod_{i:\sigma_i=1}\left(\log\,t\right)^{1/r_i}=\bar{\eta}^{-1}(t)\phi^{-1}(t)\left(\log\,t\right)^{1/r_\sigma}\leq \bar{\eta}^{-1}(t)\phi^{-1}(t)\left(\log\,t\right)^{1/r}\leq t.\]
	
	Therefore,
	\[F_{j,k}^\sigma(y)\leq C\sum_{\ell=0}^\infty \ell^{k+1-|\sigma|}(2^\ell r_j)^n\left(\prod_{i:\sigma_i=1} \|b_i-b_{i,j}^{\ell+1}\|_{\mathrm{exp}\,L^{r_i}, B_j^\ell}\right)\|K(\cdot-y)-K(\cdot-x_{Q_j})\|_{\phi, |x|\sim 2^{\ell}r_j}\|u\|_{\bar\eta, B_j^\ell}.\]
	Notice that $B_j^\ell\subset 2^{\ell+1} Q_j^*$. Let $n_0$ be the smallest integer that verifies $2^{n_0}\geq 2\sqrt{n}c$, then we have that $2^{\ell+1}Q_j^*\subset 2^{n_0+\ell+1}Q_j$. By virtue of Lemma~\ref{lema: promedios sobre cubos dilatados de b acotada por norma osc} we get
	\begin{align*}
	\prod_{i:\sigma_i=1} \left\|b_i-b_{i,j}^{\ell+1}\right\|_{\mathrm{exp}\,L^{r_i}, B_j^\ell}&\leq \prod_{i:\sigma_i=1}\left\|b_i-b_{i,j}^{\ell+1}\right\|_{\mathrm{exp}\,L^{r_i}, 2^{\ell+n_0+1}Q_j}\\
	&\lesssim \prod_{i:\sigma_i=1}\left(\left\|b_i-b_{i,j}^{\ell+n_0+1}\right\|_{\mathrm{exp}\,L^{r_i}, 2^{\ell+n_0+1}Q_j}+\left\|b_{i,j}^{\ell+n_0+1}-b_{i,j}^{\ell+1}\right\|_{\mathrm{exp}\,L^{r_i}, 2^{\ell+n_0+1}Q_j}\right)\\
	&\leq \prod_{i:\sigma_i=1}\left(\left\|b_i-b_{i,j}^{\ell+n_0+1}\right\|_{\mathrm{exp}\,L^{r_i}, 2^{\ell+n_0+1}Q_j}+Cn_0\right)\\
	&\leq C.
	\end{align*}
	Consequently,
	\begin{align*}
	F_{j,k}^\sigma(y)&\leq C\sum_{\ell=0}^\infty \ell^{k+1-|\sigma|}(2^\ell r_j)^n\|K(\cdot-y)-K(\cdot-x_{Q_j})\|_{\phi, |x|\sim 2^{\ell}r_j}\|u\|_{\bar\eta, B_j^\ell}\\
	&\leq CM_{\bar\eta} u(y)\sum_{\ell=0}^\infty \ell^{k+1-|\sigma|}(2^\ell r_j)^n\|K(\cdot-y)-K(\cdot-x_{Q_j})\|_{\phi, |x|\sim 2^{\ell}r_j}\\
	&\leq C_{k,\phi}M_{\bar\eta} u(y)\\
	&\leq C u(y),
	\end{align*}
	by Lemma~\ref{lema: control puntual de MPhi u por u} again. From these two estimates it follows that  $F_{j,k}(y)\leq Cu(y)$. Then
	\[I_1\leq C\sum_j\int_{Q_j}|h_j(y)|u(y)v(y)\,dy,\]
	and from this point we can continue the estimate in the same way as in page~\pageref{pag: estimacion de integral de h_j u v}.
	
	We can achieve the estimate of $I_2$ by applying Theorem~\ref{teo: desigualdad mixta para T Hormander} and proceeding as in page \pageref{pag: estimacion de I2}.
	
	We conclude with the estimate of $I_3$. Since $K\in H_\phi\cap H_{\eta,k+1}$ we have that $K\in H_\phi\cap H_{\eta,l}$ for each $1\leq l\leq k+1$. Fix $\sigma\in S_{k+1}$ such that $0<|\sigma|<k+1$. Then, by applying the inductive hypothesis we get
	\[uv\left(\left\{x\in \mathbb{R}^n\backslash{\Omega^*}: \frac{\left|T_{\mathbf b_\sigma}\left(\sum_j\left(\prod_{i=1}^{k+1}(b_i-\lambda_{i,j})^{\bar\sigma_i}\right)h_jv\right)(x)\right|}{v(x)}>\frac{t}{2^{k+2}}\right\}\right)\]
	is bounded by
	\[C\int_{\mathbb{R}^n}\Phi_\sigma\left(\frac{\sum_j \prod_{i=1}^{k+1}\left|b_i(x)-b_{i,j}\right|^{\bar{\sigma}_i}|h_j(x)|}{t}\right)u(x)v(x)\,dx,\]
	where $\Phi_\sigma$ is as in the proof of Theorem~\ref{teo: desigualdad mixta para Tb caso CZO}. Since $h_j$ is supported in $Q_j$, this last expression can be written as
	\begin{align*}
	I_3^1(\sigma)+I_3^2(\sigma)&=C\sum_j \int_{Q_j}\Phi_\sigma\left(\frac{\sum_j \prod_{i=1}^{k+1}\left|b_i(x)-b_{i,j}\right|^{\bar{\sigma}_i}f(x)}{t}\right)u(x)v(x)\,dx\\
	&\quad +\sum_j \int_{Q_j}\Phi_\sigma\left(\frac{\sum_j \prod_{i=1}^{k+1}\left|b_i(x)-b_{i,j}\right|^{\bar{\sigma}_i}f_{Q_j}^v(x)}{t}\right)u(x)v(x)\,dx.
	\end{align*}
	These two quantities can be estimated as in page~\pageref{pag: estimacion de I3 sigma}. This completes the proof.\qedhere	
\end{proof}

\begin{obs}
We want to point out that similar results than those contained in this article can be achieved by considering non necessarily convolution operators with the obvious changes in the hypothesis of the kernels (see for example the conditions on $K$ given in \cite{IFRR21}).
\end{obs}
\def\cprime{$'$}
\providecommand{\bysame}{\leavevmode\hbox to3em{\hrulefill}\thinspace}
\providecommand{\MR}{\relax\ifhmode\unskip\space\fi MR }
\providecommand{\MRhref}[2]{%
	\href{http://www.ams.org/mathscinet-getitem?mr=#1}{#2}
}
\providecommand{\href}[2]{#2}


\begin{thebibliography}{10}
	
	\bibitem{B-D-P}
	A.~Bernardis, E.~Dalmasso, and G.~Pradolini, \emph{Generalized maximal
		functions and related operators on weighted {M}usielak-{O}rlicz spaces}, Ann.
	Acad. Sci. Fenn. Math. \textbf{39} (2014), no.~1, 23--50.
	
	\bibitem{Berra21}
	F.~Berra, \emph{From ${A}_1$ to ${A}_\infty$: New mixed inequalities for
		certain maximal operators}, Potential Anal. (2021, in press).
	
	\bibitem{Berra-Carena-Pradolini(M)}
	F.~Berra, M.~Carena, and G.~Pradolini, \emph{Mixed weak estimates of {S}awyer
		type for commutators of generalized singular integrals and related
		operators}, Michigan Math. J. \textbf{68} (2019), no.~3, 527--564.
	
	\bibitem{CruzUribe-Martell-Perez}
	D.~Cruz-Uribe, J.~M. Martell, and C.~P\'erez, \emph{Weighted weak-type
		inequalities and a conjecture of {S}awyer}, Int. Math. Res. Not. (2005),
	no.~30, 1849--1871.
	
	\bibitem{javi}
	J.~Duoandikoetxea, \emph{Fourier analysis}, Graduate Studies in Mathematics,
	vol.~29, American Mathematical Society, Providence, RI, 2001, Translated and
	revised from the 1995 Spanish original by David Cruz-Uribe.
	
	\bibitem{GC-RdF}
	J.~Garc\'{i}a-Cuerva and J.~L. Rubio~de Francia, \emph{Weighted norm
		inequalities and related topics}, North-Holland Mathematics Studies, vol.
	116, North-Holland Publishing Co., Amsterdam, 1985, Notas de Matem\'{a}tica
	[Mathematical Notes], 104.
	
	\bibitem{IFRR21}
	Gonzalo~H. Iba\~{n}ez Firnkorn and Israel~P. Rivera-R\'{\i}os, \emph{Sparse and
		weighted estimates for generalized {H}\"{o}rmander operators and
		commutators}, Monatsh. Math. \textbf{191} (2020), no.~1, 125--173.
	
	\bibitem{KR}
	M.~A. Krasnoselski{\u\i} and J.~B. Ruticki{\u\i}, \emph{Convex functions and
		{O}rlicz spaces}, Translated from the first Russian edition by Leo F. Boron.
	P. Noordhoff Ltd., Groningen, 1961.
	
	\bibitem{L-O-P}
	K.~Li, S.~Ombrosi, and C.~P\'{e}rez, \emph{Proof of an extension of {E}.
		{S}awyer's conjecture about weighted mixed weak-type estimates}, Math. Ann.
	\textbf{374} (2019), no.~1-2, 907--929.
	
	\bibitem{LRdlT}
	M.~Lorente, J.~M. Martell, M.~S. Riveros, and A.~de~la Torre, \emph{Generalized
		{H}\"{o}rmander's conditions, commutators and weights}, J. Math. Anal. Appl.
	\textbf{342} (2008), no.~2, 1399--1425.
	
	\bibitem{Muck72}
	B.~Muckenhoupt, \emph{Weighted norm inequalities for the {H}ardy maximal
		function}, Trans. Amer. Math. Soc. \textbf{165} (1972), 207--226.
	\MR{0293384}
	
	\bibitem{Perez95}
	C.~P\'erez, \emph{Endpoint estimates for commutators of singular integral
		operators}, J. Funct. Anal. \textbf{128} (1995), no.~1, 163--185.
	
	\bibitem{Perez-95-Onsuf}
	C.~P\'{e}rez, \emph{On sufficient conditions for the boundedness of the
		{H}ardy-{L}ittlewood maximal operator between weighted {$L^p$}-spaces with
		different weights}, Proc. London Math. Soc. (3) \textbf{71} (1995), no.~1,
	135--157.
	
	\bibitem{Perez_Sharp97}
	C.~P\'erez, \emph{Sharp estimates for commutators of singular integrals via
		iterations of the {H}ardy-{L}ittlewood maximal function}, J. Fourier Anal.
	Appl. \textbf{3} (1997), no.~6, 743--756. \MR{1481632}
	
	\bibitem{P-T}
	C.~P\'{e}rez and R.~Trujillo-Gonz\'{a}lez, \emph{Sharp weighted estimates for
		multilinear commutators}, J. London Math. Soc. (2) \textbf{65} (2002), no.~3,
	672--692.
	
	\bibitem{raoren}
	M.~M. Rao and Z.~D. Ren, \emph{Theory of {O}rlicz spaces}, Monographs and
	Textbooks in Pure and Applied Mathematics, vol. 146, Marcel Dekker, Inc., New
	York, 1991. \MR{1113700}
	
	\bibitem{Sawyer}
	E.~Sawyer, \emph{A weighted weak type inequality for the maximal function},
	Proc. Amer. Math. Soc. \textbf{93} (1985), no.~4, 610--614.
	
\end{thebibliography}
\end{document}